\documentclass[11pt,reqno]{amsart}
\usepackage{amssymb,epsfig,graphics}

\newtheorem{theorem}{Theorem}[section]
\newtheorem{hypothesis}[theorem]{Hypothesis}
\newtheorem{proposition}[theorem]{Proposition}

\newtheorem{corollary}[theorem]{Corollary}

\newtheorem{sub-lemma}[theorem]{Sub-Lemma}
\newtheorem{remark}[theorem]{Remark}

\def\Z{\mathcal{Z}}

\def\D{\Delta}

\def\RR{\mathbb{R}}

\def\ZZ{\mathbb{Z}}
\def\CC{\mathbb{C}}

\let\eps=\varepsilon

\def\glu{\!\!\!\!\!\!\!}
\def\gluu{\glu\glu\glu}

\def\glu{\!\!\!}
\def\B{\mathcal{B}}
\def\D{\mathcal{D}}

\def\Z{\mathcal{Z}}
\def\RR{{\mathbb R}}
\def\1{{{\mathit 1} \!\!\>\!\! I} }

\DeclareMathOperator{\esssup}{esssup}

\begin{document}

\title[Decorrelation in infinite measure]{Mixing and decorrelation in infinite measure: the case of the periodic Sinai billiard}
\author{Fran\c{c}oise P\`ene}
\address{1) Universit\'e de Brest, Laboratoire de
Math\'ematiques de Bretagne Atlantique, CNRS UMR 6205, France\\
2) Institut Universitaire de France\\
3) Universit\'e de Bretagne Loire}
\email{francoise.pene@univ-brest.fr}
\keywords{Sinai, billiard, Lorentz process,
Young tower, local limit theorem, decorrelation, mixing, infinite measure}
\subjclass[2000]{Primary: 37A25}
\begin{abstract}
We investigate the question of the rate of mixing for observables of a $\mathbb Z^d$-extension of a probability preserving 
dynamical system with good spectral properties.
We state general mixing results, including expansions of every order. The main part of this article is devoted to the study of mixing rate for
smooth observables of
the $\mathbb Z^2$-periodic Sinai billiard, with different kinds of results
depending on whether the horizon is finite or infinite. We establish a first order mixing
result when the horizon is infinite. In the finite horizon case,
we establish an asymptotic expansion of every order, enabling the study of the mixing rate even for observables with null integrals.
\end{abstract}
\date{\today}
\maketitle
\bibliographystyle{plain}
\section*{Introduction}
Let $(M,\nu,T)$ be a dynamical system, that is a measure space $(M,\nu)$ endowed with a measurable transformation $T:M\rightarrow M$ which preserves the measure $\nu$. The mixing properties deal with the asymptotic behaviour, as $n$ goes to infinity, of integrals of the following form
$$C_n(f,g):=\int_Mf.g\circ T^n\, d\nu,$$
for suitable
observables $f,g:M\rightarrow \mathbb C$.

Mixing properties of probability preserving dynamical systems have been studied by many authors.
It is a way to measure how chaotic the dynamical
system is. A probability preserving dynamical system is said to be mixing if  $C_n(f,g)$
converges to $\int_Mf\, d\nu\, \int_Mg\, d\nu$ for every square integrable observables $f,g$. When a probability preserving system is mixing, a natural question is to study the decorrelation rate, i.e.  
the rate at which 
$C_n(f,g)$
converges to zero when $f$ or $g$ have null expectation. This crucial question is often a first step before proving probabilistic limit theorems (such as central limit theorem and its variants). The study of this question has a long history. Such decays of covariance have been studied for wide classes of smooth observables $f,g$ and for many probability preserving dynamical systems. In the case of the Sinai billiard, such results and further properties have been established in \cite{Sinai70,BS80,BS81,BCS90,BCS91,young,Chernov,SV1,SV2}.

We are interested here in the study of mixing properties when the
invariant measure $\nu$ is $\sigma$-finite.
In this context, as noticed in \cite{KS}, there is no satisfactory notion of mixing. Nevertheless the question of the rate of mixing for smooth observables is natural. A first step in this direction is to establish results of the following form:
\begin{equation}\label{MIXING}
\lim_{n\rightarrow +\infty}{\alpha_n}C_n(f,g)=\int_Mf\, d\nu \, \int_Mg\, d\nu\, .
\end{equation}
Such results have been proved in \cite{Thaler,MT,Gouezel,BT,LT}
for a wide class of models and for smooth functions $f,g$, using  induction on a finite measure subset of $M$.

An alternative approach, specific to the case of $\mathbb Z^d$-extensions of probability preserving dynamical system,
has been pointed out in \cite{FP17}. The idea therein is that,
in this particular context, \eqref{MIXING}
is related to a precised local limit theorem.
In the particular case of the $\mathbb Z^2$-periodic Sinai billiard with finite horizon, it
has been proved in \cite{FP17} that
$$C_n(f,g)=\frac {c_0} n\int_M f\, d\nu\, \int_M g\, d\nu+o(n^{-1})\, , $$
for some explicit constant $c_0$, for some dynamically Lipschitz functions, including functions with full support in $M$.

This paper is motivated by the question of high order expansion
of mixing and by the study of the mixing rate for observables
with null integrals. This last question can be seen as decorrelation rate
in the infinite measure.
Let us mention the fact that it has been proved  in \cite{damiensoaz}, for the billiard in finite horizon, that sums
$\sum_{k\in\mathbb Z}\int_Mf.f\circ T^k\, d\nu$ are well defined for
some observables $f$ with null expectation.
In the present paper, 
we use the approach of \cite{FP17} to establish, in the context of the $\mathbb Z^2$-periodic Sinai billiard with finite horizon, a high order mixing result of the following form:
\begin{equation}\label{DEVASYMP}
C_n(f,g)=\sum_{m=0}^{K-1}\frac{c_m(f,g)}{n^{1+m}}+o(n^{-K}) \, .
\end{equation}
This estimate enables the study of the rate of convergence of 
$nC_n(f,g)$ to $\int_Mf\, d\nu\, \int_Mg\, d\nu$ and, most importantly,
it enables the study of the rate
of decay of $C_n(f,g)$ for functions $f$ or $g$ with integral 0.
In general, if $f$ or $g$ have zero integral we have
$$C_n(f,g)\sim\frac{c_1(f,g)}{n^{2}}\, ,$$
but it may happen that
$$C_n(f,g)\sim\frac{c_2(f,g)}{n^{3}} \, ,$$
and even that $C_n(f,g)=o(n^{-3})$.
For example, \eqref{DEVASYMP} gives immediately that, if $\int_Mf\, d\nu\int_Mg\, d\nu\ne 0$, then
\begin{eqnarray}
C_n(f-f\circ T,g)&=&C_n(f,g)-C_{n-1}(f,g)\nonumber\\
&\sim&-c_0\frac{\int_Mf\, d\nu.\int_Mg\, d\nu}{n^2}=\frac{c_1(f-f\circ T,g)}{n^2}\label{Cncobg}
\end{eqnarray}
and
\begin{eqnarray*}
C_n(2f-f\circ T-f\circ T^{-1},g)&=&C_n(f-f\circ T,g-g\circ T)\\
  &=&2C_n(f,g)-C_{n-1}(f,g)-C_{n+1}(f,g)\\
  &\sim& -\frac {2c_0}{n^3}\int_Mf\, d\nu
  \int_Mg\, d\nu=\frac{c_2(f-f\circ T,g-g\circ T)}{n^3}\, .
\end{eqnarray*}
General formulas for the dominating term will be given in Theorem \ref{MAIN}, Remark \ref{RQE} and Corollary \ref{coroMAIN}. In particular $c_1(f,g)$ and $c_2(f,g)$ will be precised.

We point out the fact that the method we use is rather general in the context of $\mathbb Z^d$-extensions over dynamical systems with good spectral properties, and that, to our knowledge, these are the first results of this kind for dynamical systems preserving an infinite measure.\medskip

We establish moreover an estimate of the following form
for 
smooth observables of the $\mathbb Z^2$-periodic Sinai billiard with infinite horizon: 
$$C_n(f,g)=\frac {c_0} {n\log n}\int_M f\, d\nu\, \int_M g\, d\nu+o((n\log n)^{-1})\, . $$

The paper is organized as follows. In Section \ref{sec:model}, we 
present the model of the $\mathbb Z^2$-periodic Sinai billiard
and we state our main results for this model (finite/infinite horizon).
In Section \ref{sec:GENE}, we state general mixing results for $\mathbb Z^d$-extensions of probability preserving dynamical systems for which the Nagaev-Guivarc'h perturbation method can be implemented.
In Section \ref{sec:young}, we recall some facts on the towers constructed by Young for the Sinai billiards.
In Section \ref{sec:MAIN}, we prove our main results for the billiard
in finite horizon (see also Appendix \ref{sec:coeff} for the computation of the first coefficients). In Section \ref{sec:infinite},
we prove our result for the billiard in infinite horizon.

\section{Main results for $\mathbb Z^2$-periodic Sinai billiards}\label{sec:model}
Let us introduce the $\mathbb Z^2$-periodic Sinai billiard
$(M,\nu,T)$.

Billiards systems modelise the behaviour of a point particle
moving at unit speed in a domain $Q$ and bouncing off $\partial Q$ with respect to the Descartes reflection law (incident angle=reflected angle). 
We assume here that $Q:=\mathbb R^2\setminus\bigcup_{\ell\in\mathbb Z^2}\bigcup_{i=1}^I (O_i+\ell)$, with $I\ge 2$ and where $O_1,...,O_I$ are convex bounded open sets (the boundaries of which are $C^3$-smooth and have non null curvature).
We assume that the closures of the obstacles $O_i+\ell$ are pairwise
disjoint. 
The billiard is said to have {\bf finite horizon} if
every line in $\mathbb R^2$ meets $\partial Q$. Otherwise it is said to have {\bf infinite horizon}.

We consider the dynamical system $(M,\nu,T)$ corresponding
to the dynamics at reflection times which is defined as follows.
Let $M$ be the set of reflected vectors off $\partial Q$, i.e. 
$$M:=\{(q,\vec v)\in\partial Q\times S^1\ :\ \langle \vec n(q),\vec v\rangle\ge 0\},$$
where $\vec n(q)$ stands for the unit normal vector to $\partial Q$
at $q$ directed inward $Q$. We decompose this set into $M:=\bigcup_{\ell\in\mathbb Z^2}\mathcal C_\ell$, with
$$\mathcal C_\ell:=\left\{(q,\vec v)\in M\ :\ q\in\bigcup_{i=1}^I (\partial O_i+\ell)\right\}.$$
The set $\mathcal C_\ell$ is called the $\ell$-cell.
We define $T:M\rightarrow M$ as the transformation mapping a reflected vector at a reflection time to the reflected vector at the next reflection time. 
We consider the measure $\nu$ absolutely continuous with respect to the Lebesgue measure on $M$, with density
proportional to $(q,\vec v)\mapsto \langle \vec n(q),\vec v\rangle$ and such that $\nu(\mathcal C_{0})=1$.

Because of the $\mathbb Z^2$-periodicity of the model, there
exists a transformation $\bar T:\mathcal C_{ 0}\rightarrow\mathcal C_{ 0}$ and a function
$\kappa:\mathcal C_{ 0}\rightarrow\mathbb Z^2$ such that
\begin{equation}\label{skewproduct}
\forall ((q,\vec v),\ell)\in\mathcal C_{ 0}\times\mathbb Z^2,\ 
    T(q+\ell, \vec v)=\left(q'+\ell+\kappa(q,\vec v),\vec v'\right),\ 
\mbox{if}\ \bar T (q,\vec v)=(q',\vec v').
\end{equation}
This allows us to define a probability preserving dynamical $(\bar M,\bar\mu,\bar T)$ (the Sinai billiard)
by setting $\bar M:=\mathcal C_{ 0}$ and $\bar\mu=\nu_{|\mathcal C_{ 0}}$.
Note that \eqref{skewproduct} means that $(M,\nu,T)$ can be represented by the $\mathbb Z^2$-extension of $(\bar M,\bar\mu,\bar T)$ by $\kappa$. In particular, iterating \eqref{skewproduct}
leads to
\begin{equation}\label{skewproductn}
\forall ((q,\vec v),\ell)\in\mathcal C_{0}\times\mathbb Z^2,\ 
    T^n(q+\ell, \vec v)=\left(q'_n+\ell+S_n(q,\vec v),\vec v'_n\right),
\end{equation}
if $\bar T^n (q,\vec v)=(q'_n,\vec v'_n)$ and with the notation
$$ S_n:=\sum_{k=0}^{n-1}\kappa\circ \bar T^k.$$
The set of tangent reflected vectors $\mathcal S_0$ given by
$$\mathcal S_0:=\{(q,\vec v)\in M\ :\ \langle \vec v,\vec n(q)\rangle=0\} $$
plays a special role in the study of $T$.
 Note that $T$ defines a $C^1$-diffeomorphism from $M\setminus
(\mathcal S_0\cup T^{-1}(\mathcal S_0))$ to $M\setminus
(\mathcal S_0\cup T(\mathcal S_0))$.

Statistical properties of $(\bar M,\bar\mu,\bar T)$ have
been studied by many authors since the seminal article \cite{Sinai70} by Sinai. 

In the finite horizon case, limit theorems have been established
in \cite{BS81,BCS91,young,Chernov}, including the convergence
in distribution of $(S_n/\sqrt{n})_n$ to a centered gaussian random variable $B$
with nondegenerate variance matrix $\Sigma^2$ given by:
$$\Sigma^2:=\sum_{k\in\mathbb Z}\mathbb E_{\bar\mu}[\kappa\otimes\kappa\circ \bar T^k]\, ,$$
where we used the notation $X\otimes Y$ for the matrix $(x_iy_j)_{i,j}$, for $X=(x_i)_i,Y=(y_j)_j\in\mathbb C^2$.
Moreover a local limit theorem for $S_n$ has been established in \cite{SV1} and some of its refinements have been stated and used in \cite{DSV,FP09a,FP09b,ps10} with various applications. Recurrence and ergodicity of this model follow from \cite{JPC,Schmidt,SV1,Simanyi,FP00}.

In the infinite horizon case, a result of exponential decay of correlation has been proved in \cite{Chernov}.
A nonstandard central limit theorem (with normalization in $\sqrt{n\log n}$) and a local limit theorem have been established in \cite{SV2}, ensuring recurrence and ergodicity of the infinite measure system $(M,\nu,T)$. 
This result states in particular that $(S_n/\sqrt{n\log n})_n$
converges in distribution to a centered gaussian distribution with
variance $\Sigma_\infty^2$ given by
$$\Sigma_\infty^2:=\sum_{x\in \mathcal S_0|\bar T x=x}\frac{d_x^2}{2|\kappa(x)|\, \sum_{i=1}^I|\partial O_i|}(\kappa(x))^{\otimes 2}\, ,$$
where $d_x$ is the width of the corridor corresponding to $x$.

Our main results provide mixing estimates for dynamically Lipschitz
functions. Let us introduce this class of observables.
Let $\xi\in(0,1)$. We consider the metric $d_\xi$ on $M$ given by
$$\forall x,y\in M,\quad d_\xi(x,y):=\xi^{s(x,y)},$$
where $s$ is a separation time defined as follows:
$s(x,y)$ is the maximum of the integers $k\ge 0$ such that $x$ and $y$
lie in the same connected component of $M\setminus \bigcup_{j=-k}^kT^{-j}\mathcal S_0$. For every $f:M\rightarrow \mathbb C$, we write $L_\xi(f)$ for the Lipschitz constant with respect to $d_\xi$:
$$L_\xi(f):=\sup_{x\ne y}\frac{|f(x)-f(y)|}{d_\xi(x,y)}\, . $$
We then set
$$\Vert f\Vert_{(\xi)}:=\Vert f\Vert_\infty+L_\xi(f)\, . $$
Before stating our main result, let us introduce some additional notations.

We will work with symmetric multilinear forms.
For any
$A=(A_{i_1,...,i_m})_{(i_1,...,i_m)\in\{1,2\}^m}$ and
$B=(B_{i_1,...,i_k})_{(i_1,...,i_k)\in\{1,2\}^k}$ 
with complex entries ($A$ and $B$ are identified respectively with a $m$-multilinear
form on $\mathbb C^2$ and with a $k$-multilinear form on $\mathbb C^2$), we define $A\otimes B$ as the element
$C$ of $\mathbb C^{\{1,2\}^{m + m'}}$ (identified with
a $(m+ m')$-multilinear form on $\mathbb C^2$)
such that
$$\forall i_1,,...,i_{m+m'}\in\{1,2\},\quad C_{(i_1,,...,i_{m+m'})}=A_{(i_1,...,i_{m})}
B_{(i_{m+1}...,i_{m+m'})} .$$
For any
$A=(A_{i_1,...,i_m})_{(i_1,...,i_m)\in\{1,2\}^m}$ and
$B=(B_{i_1,...,i_k})_{(i_1,...,i_k)\in\{1,2\}^k}$ symmetric
with complex entries with $k\le m$, we define $A* B$ as the element
$C$ of $\mathbb C^{\{1,2\}^{m-k}}$ (identified with
a $(m-k)$-multilinear form on $\mathbb C^2$) such that
$$\forall i_1,,...,i_{m-k}\in\{1,2\},\quad C_{(i_1,,...,i_{m-k})}=\sum_{i_{m-k+1},...,i_m\in\{1,2\}}A_{(i_1,...,i_{m})}
B_{(i_{m-k+1},...,i_{m})} .$$
We identify naturally vectors in $\mathbb C^2$ with $1$-linear functions and symmetric matrices with symmetric bilinear functions.
For any $C^m$-smooth function $F:\mathbb C^2\rightarrow\mathbb C$, we write $F^{(m)}$ for its $m$-th
differential, which is identified with a $m$-linear function on $\mathbb C^2$. 
We write $A^{\otimes k}$ for the product $A\otimes...\otimes A$.
Observe that, with these notations, Taylor expansions of $F$ at $0$ are simply written
$$\sum_{k=0}^m F^{(k)}(0)*x^{\otimes k}\,  .$$
It is also worth noting that $A* (B\otimes C)=(A*B)*C$, for every $A,B,C$ corresponding to symmetric multilinear forms with respective ranks $m,k,\ell$ with $m\ge k+\ell$.

We extend the definition of $\kappa$ to $M$ by setting $\kappa((q+\ell,\vec v))=\kappa(q,\vec v)$ for every $(q,\vec v)\in\bar M$ and every $\ell\in\mathbb Z^2$.
For every $k\in\mathbb Z$ and every $x\in M$, we write $\mathcal I_k(x)$ for the label in $\mathbb Z^2$ of the cell containing $T^kx$, i.e. $\mathcal I_k$ is the label of the cell in which the particle is at the $k$-th
reflection time. It is worth noting that, for $n\ge 0$, we have $\mathcal I_n-\mathcal I_0=\sum_{k=0}^{n-1}\kappa\circ T^k$ and
$\mathcal I_{-n}-\mathcal I_0=-\sum_{k=-n}^{-1}\kappa\circ T^k$.
\medskip

Now let us state our main results, the proofs of which are postponed to
Section \ref{sec:MAIN}.
We start by stating our result in the infinite horizon case, and then we will present sharper results in the finite horizon case.

\subsection{$\mathbb Z^2$-periodic Sinai billiard with infinite horizon}
\begin{theorem}
\label{horizoninfini}
Let $(M,\nu, T)$ be the $\mathbb Z^2$-periodic Sinai billiard with infinite horizon.
Suppose that the set of corridor free flights $\{\kappa(x),\ x\in \mathcal S_0,\ \bar T x=x\}$ spans $\mathbb R^2$.
Let $f,g:M\rightarrow\mathbb C$ (with respect to $d_\xi$) 
be two dynamically Lipschitz continuous functions such that
\begin{equation}\label{toto}
\sum_{\ell\in\mathbb Z^2}\left(\Vert f
      \mathbf 1_{\mathcal C_\ell}\Vert_{\infty}+\Vert g
      \mathbf 1_{\mathcal C_\ell}\Vert_{\infty}\right)<\infty\, .
\end{equation}
Then
$$\int_Mf.g\circ T^n\, d\nu=\frac{1}{2\pi\sqrt{\det\Sigma_\infty^2}\, n\log n}\left(\int_Mf\, d\nu\, \int_Mg\, d\nu+o(1)\right)\, .$$
\end{theorem}

\subsection{$\mathbb Z^2$-periodic Sinai billiard with finite horizon}
We first state our result providing an expansion of every order for the mixing
(see Theorem \ref{MAIN} and Corollary \ref{coroMAIN} for more details).
\begin{theorem}\label{PRINCIPAL}
Let $K$ be a positive integer.
Let $f,g:M\rightarrow\mathbb C$ be two dynamically Lipschitz continuous observables
such that
\begin{equation*}
\sum_{\ell\in\mathbb Z^2}|\ell|^{2K-2}(\Vert f\mathbf 1_{\mathcal C_\ell}\Vert_{(\xi)}+\Vert g\mathbf 1_{\mathcal C_\ell}\Vert_{(\xi)})<\infty\, ,
\end{equation*}
then there exist $c_0(f,g),...,c_{K-1}(f,g)$
such that
\[
\int_Mf.g\circ T^n\, d\nu
=\sum_{m=0}^{K-1}\frac{c_m(f,g)}{n^{1+m}}+o(n^{-K})\, .
\]
\end{theorem}
We precise in the following theorem the expansion of order 2.
\begin{theorem}\label{MAINbis}
Let $f,g:M\rightarrow\mathbb R$ be two bounded observables
such that
$$\sum_{\ell\in\mathbb Z^2}|\ell|^{2}(\Vert f\mathbf 1_{\mathcal C_\ell}\Vert_{(\xi)}+\Vert g\mathbf 1_{\mathcal C_\ell}\Vert_{(\xi)})<\infty\, .$$
Then
\begin{eqnarray}
\int_M f.g\circ T^n\, d\nu
&=&\frac 1{2\pi\sqrt{\det\Sigma^2}}\left\{\frac{1}{n}\int_M f\, d\nu \, \int_M g\, d\nu
+ \frac {1}{2\, n^2}\, \Sigma^{-2}*\tilde {\mathfrak A}_2(f,g)\nonumber\right.\\
&\ &
\left.+\frac 1{4!\, n^2}\int_Mf\, d\nu \, \int_Mg\, d\nu \,  (\Sigma^{-2})^{\otimes 2}* \Lambda_4\right\}
+o(n^{-2})\, , \label{MAIN2}
\end{eqnarray}
with $\Sigma^{-2}=(\Sigma^2)^{-1}$ and
\[
\tilde {\mathfrak A_2}(f,g):= 
-\int_Mf\, d\nu \, \mathfrak B_2^-(g)-\int_Mg\, d\nu\, \mathfrak B_2^+(f)-\int_Mf\, d\nu\int_Mg\, d\nu\, \mathfrak B_0+2\, \mathfrak B_1^+(f)\otimes\mathfrak B_1^-(g)\, ,
\]
\[
\mathfrak B_2^+(f):=\lim_{m\rightarrow +\infty} \int_Mf.\left(\mathcal I_m^{\otimes 2}-m\Sigma^2\right)\, d\nu \, ,
\]
\[
\mathfrak B_2^-(g):=\lim_{m\rightarrow -\infty}\int_M g.\left(\mathcal I_m^{\otimes 2}-|m|\Sigma^2\right)\, d\nu\, ,
\]
\[
\mathfrak B_1^+(f):=\lim_{m\rightarrow +\infty}\int_M f.\mathcal I_m\, d\nu \, ,
\quad
\mathfrak B_1^-(g):=\lim_{m\rightarrow -\infty}\int_M g.\mathcal I_m\, d\nu\, ,
\]
\[
\mathfrak B_0=\lim_{m\rightarrow +\infty}(m\Sigma^2-\mathbb E_{\bar\mu}[S_m^{\otimes 2}])
\]
and
\[
\Lambda_4:=\lim_{n\rightarrow +\infty}\frac{\mathbb E_{\bar\mu}[S_n^{\otimes 4}]-3n^2(\Sigma^2)^{\otimes 2}}{n}+6\Sigma^2\otimes \mathfrak B_0\, .
\]
\end{theorem}
Observe that we recover \eqref{Cncobg} since $\Sigma^2*\Sigma^{-2}=2$,
$$\mathfrak B_1^+(f-f\circ T)=\lim_{m\rightarrow +\infty}\int_Mf.\kappa\circ T^m\, d\nu=0 $$
and
\begin{eqnarray*}
\mathfrak B_2^+(f-f\circ T)&=&\lim_{n\rightarrow +\infty}\int_Mf.(\mathcal I_m^{\otimes 2}-\mathcal I_{m-1}^{\otimes 2})\\
&=&\lim_{m\rightarrow +\infty}\int_Mf.\left(\kappa^{\otimes 2}\circ T^{m-1}+2\sum_{k=0}^{m-2}(\kappa\circ T^k)\otimes\kappa\circ T^{m-1}\right)\, d\nu \\
&=&\lim_{m\rightarrow +\infty}\int_Mf\, d\nu \mathbb E_{\bar\mu}\left[\kappa^{\otimes 2}+2\sum_{k=1}^{m-1}\kappa\otimes\kappa\circ T^{k}\right]\, ,\\
&=&\Sigma^2\int_Mf\, d\nu \, ,
\end{eqnarray*}
where we used Proposition \ref{decoChernov}.

\begin{remark}
Note that
\begin{eqnarray*}
\mathfrak B_2^+(f)
&=&  \sum_{j,m\ge 0}\int_M f .\left(\kappa\circ T^j\otimes\kappa\circ T^m
    -\mathbb E_{\bar\mu}[\kappa\circ \bar T^j\otimes\kappa\circ\bar T^m]\right)\, d\nu\nonumber\\
&\ &+\int_M f\mathcal I_0^{\otimes 2}\, d\nu+2\sum_{m\ge 0}\int_M f.\mathcal I_0\otimes  \kappa\circ \bar T^m\, d\nu-\mathfrak B_0\int_Mf\, d\nu\, ,
\end{eqnarray*}
\begin{eqnarray*}
\mathfrak B_2^-(g)&=& \sum_{j,m\le -1}
\int_M g.(\kappa\circ T^j\otimes \kappa\circ T^m-\mathbb E_{\bar\mu}[\kappa\circ \bar T^j\otimes \kappa\circ\bar T^m])\, d\nu\\
&\ & +\int_M g.\mathcal I_0^{\otimes 2}\, d\nu
        -2 \sum_{m\le -1}\int_M g.\mathcal I_0\otimes\kappa \circ T^m\, d\nu-\mathfrak B_0\int_Mg\, d\nu\, ,
\end{eqnarray*}
\[
\mathfrak B_1^+(f)=\sum_{m\ge 0}\int_M f.\kappa\circ T^m\, d\nu+
\int_M f.\mathcal I_0\, d\nu \, ,
\]
\[
\mathfrak B_1^-(g)=-\sum_{m\le -1}\int_M g.\kappa\circ T^m\, d\nu+
\int_M g.\mathcal I_0\, d\nu \, ,
\]
and
\[
\mathfrak B_0=\sum_{m\in\mathbb Z}|m|\mathbb E_{\bar\mu}[\kappa\otimes\kappa\circ\bar T^m]\, .
\]
\end{remark}
\begin{corollary}
Under the assumptions of Theorem \ref{MAINbis}, if $\int_Mf\, d\nu=0$ and 
$\int_Mg\, d\nu=0$, then
$$\int_Mf.g\circ T^n\, d\nu =\frac{\Sigma^{-2}*(\mathfrak B_1^+(f)\otimes \mathfrak B_1^-(g))}
{n^2\,2\pi\sqrt{\det\Sigma^2}}
+o(n^{-2})\, .$$
\end{corollary}
Two natural examples of zero integral functions are $\mathbf 1_{\mathcal C_0}-\mathbf 1_{\mathcal C_{e_1}}$ with $e_1=(1,0)$ or $f\mathcal C_0$ with $\int_{\mathcal C_0}f\, d\nu=0$. Note that
$$\int_M((\mathbf 1_{\mathcal C_0}-\mathbf 1_{\mathcal C_{e_1}}).(\mathbf 1_{\mathcal C_0}-\mathbf 1_{\mathcal C_{e_1}})\circ T^n\mathbf)\, d\nu \sim \frac{\sigma^2_{2,2}}{n^2\,2\pi({\det\Sigma^2})^{3/2}},$$
with $\Sigma^2=(\sigma^2_{i,j})_{i,j=1,2}$ and that
$$\int_M(f\mathbf 1_{\mathcal C_0}.\mathbf 1_{\mathcal C_0}\circ T^n\mathbf)\, d\nu \sim -\frac{1}{n^2\,2\pi({\det\Sigma^2})^{3/2}}\sum_{m\ge 0}
    \mathbb E_{\bar\mu}[f.(\sigma^2_{2,2}\kappa_1+\sigma^2_{1,1}\kappa_2)\circ T^m]\, ,$$
with $\kappa=(\kappa_1,\kappa_2)$, provided the sum appearing in the last formula is non null.
As noticed in introduction, it may happen
that \eqref{MAIN2} provides only $\int_M f.g\circ T^n=o(n^{-2})$.
This is the case for example if $\int_Mg\, d\nu=0$ and if $f$ has the form 
$f(q+\ell,\vec v)=f_0(q,\vec v).h_\ell$ with $\mathbb E_{\bar\mu}[f_0]=0$ and $\sum_\ell h_\ell=0$.

Hence it can be useful to go further in the asymptotic expansion, which is possible thanks to Theorem \ref{MAIN}. A formula for the
term of order $n^{-3}$ when $\int_M f\, d\nu=\int_Mg\, d\nu=\tilde{\mathfrak A}_2(f,g)=0$ is stated in theorem \ref{MAINter}
and gives the following estimate, showing that, for some
observables, $C_n(f,g)$ has order $n^{-3}$.
\begin{proposition}\label{casparticulier}
If $f$ and $g$
can be decomposed in $f(q+\ell,\vec v)=f_0(q,\vec v).h_\ell$
and $g(q+\ell,\vec v)=g_0(q,\vec v).q_\ell$ with $\mathbb E_{\bar\mu}[f_0]=\mathbb E_{\bar\mu}[g_0]=0$ and $\sum_\ell q_\ell=\sum_\ell h_\ell=0$ such that $\sum_{\ell\in\mathbb Z^2}|\ell|^4(\Vert f\mathbf 1_{\mathcal C_\ell}\Vert_{(\xi)}+\Vert g\mathbf 1_{\mathcal C_\ell}\Vert_{(\xi)})<\infty$. Then
$$\int_Mf.g\circ T^n\, d\nu=\frac {(\Sigma^{-2})^{\otimes 2}} {2\pi\sqrt{\det \Sigma^2}n^3}*\frac{\mathfrak B_2^+(f)\otimes
\mathfrak B_2^-(g)}4+o(n^{-3})\, ,$$
with here 
$$
\frac{\mathfrak B_2^+(f)\otimes
\mathfrak B_2^-(g)}4=
-\left(\sum_{\ell\in\mathbb Z^2} h_\ell.\ell\right)\otimes
\left(\sum_{j\ge 0}\mathbb E_{\bar\mu}[f_0.\kappa\circ T^j]\right)\otimes\left(\sum_{\ell\in\mathbb Z^2} q_\ell.\ell\right)
   \otimes \left(\sum_{m\le -1}\mathbb E_{\bar\mu}[g_0.\kappa\circ T^m]\right)\, .$$
\end{proposition}
\section{General results for $\mathbb Z^d$-extensions and key ideas}\label{sec:GENE}
In this section we state general results in the general context of $\mathbb Z^d$-extensions over dynamical systems satisfying good spectral
properties. This section contains the rough ideas of the proofs for the billiard, without some complications due to the quotient tower. Moreover
the generality of our assumptions makes our results implementable to a wide class of models with present and future developments of the Nagaev-Guivarch method of perturbation of transfer operators.\medskip

We consider a dynamical system $(M,\nu,T)$ given by the $\mathbb Z^d$-extension of a probability preserving dynamical system $(\bar M,\bar\mu,\bar T)$ by $\kappa:\bar M\rightarrow\mathbb Z^d$.
This means that $M=\bar M\times\mathbb Z^d$, $\nu=\bar\mu\otimes\mathfrak m_d$ where $\mathfrak m_d$ is the counting measure on $\mathbb Z^d$ and with
$$\forall (x,\ell)\in\bar M\times\mathbb Z^d,\quad T(x,\ell)=(\bar T(x),\ell+\kappa(x))\, , $$
so that
$$\forall (x,\ell)\in\bar M\times\mathbb Z^d,\ \forall n\ge 1,\quad T^n(x,\ell)=(\bar T^n(x),\ell+S_n(x))\, , $$
with $S_n:=\sum_{k=0}^{n-1}\kappa\circ\bar T^k$.
Let $P$ be the transfer operator of $\bar T$, i.e. the dual operator
of $f\mapsto f\circ\bar T$.
Our method is based on the following key fomulas:
\begin{eqnarray}
\int_M f.g\circ T^n\, d\nu&=&\sum_{\ell,\ell'\in\mathbb Z^2}\mathbb E_{\bar\mu}[f(\cdot,\ell).\mathbf 1_{S_n=\ell'-\ell}.g(\bar T^n(\cdot),\ell')]\label{FORMULECLEF0}\\
&=&\sum_{\ell,\ell'\in\mathbb Z^d}
\mathbb E_{\bar\mu}[P^n(\mathbf 1_{S_n=\ell'-\ell}\, f(\cdot,\ell))g(\cdot,{\ell'})]\, 
\label{FORMULECLEF}
\end{eqnarray}
and
\begin{eqnarray}
P^n(\mathbf 1_{S_n=\ell}\, u)&=&\frac 1{(2\pi)^d}\int_{[-\pi,\pi]^d}e^{-it*\ell}  P^n(e^{it*S_n}u)\, dt\nonumber\\
&=&\frac 1{(2\pi)^d}\int_{[-\pi,\pi]^d}e^{-it*\ell}
P_t^n(u)\, dt\, ,    \label{formulespectrale}
\end{eqnarray}
with $P_t:=P(e^{it*\kappa}\cdot)$.
Note that \eqref{FORMULECLEF} makes a link between 
mixing properties and the local limit theorem and that \eqref{formulespectrale} shows the importance of the study of the family of perturbed operators $(P_t)_t$ in this study.

We will make the following general assumptions about $(P_t)_t$. 
\begin{hypothesis}[Spectral hypotheses]\label{HHH}
There exist two complex Banach spaces $(\mathcal B,\Vert{\cdot}\Vert)$ and $(\mathcal B_0,\Vert\cdot\Vert_0)$ such that:
\begin{itemize}
\item $\mathcal B \hookrightarrow \mathcal B_0\hookrightarrow  L^1 (\bar M, \bar \mu)$ and 
$\mathbf{1}_{\bar M} \in \mathcal B$ ,
\item there exist constants $b\in(0,\pi]$, $C>0$ and $\vartheta \in(0,1)$ and three functions $\lambda_\cdot:[-b,b]^d \to \mathbb C$ and
$\Pi_\cdot,R_\cdot:[-b,b]^d \to \mathcal L(\mathcal B,\mathcal B)$ 
such that $\lim_{t\rightarrow 0}\lambda_t=1$ and $\lim_{t\rightarrow 0}\Vert \Pi_t-\mathbb E_\mu [\cdot]\mathbf 1_{\bar M}\Vert_{\mathcal L(\mathcal B,\mathcal B_0)}=0$ and such that, in $\mathcal L(\mathcal B,\mathcal B)$,
\begin{equation}\label{decomp}
\forall u\in[-b,b]^d,\quad P_u =\lambda_u\Pi_u+R_u,\quad 
\Pi_u R_u      = R_u \Pi_u = 0,\quad
\Pi_u^{2}  = \Pi_{u}\, ,
\end{equation}
\begin{equation}
\sup_{u\in [-b,b]^d} \Vert{R_u^k}\Vert_{\mathcal L(\mathcal B,\mathcal B_0)} \leq C \vartheta^k,\quad\sup_{u\in[-\pi,\pi]^d \setminus [-b,b]^d} \Vert{P_u^k}\Vert_{\mathcal L(\mathcal B,\mathcal B_0)}  \leq C \vartheta	^k. 	     \end{equation}
\end{itemize}
\end{hypothesis}
Note that \eqref{decomp} ensures that
\begin{equation}\label{decomp2}
\forall u\in[-b,b],\quad  P_u^n =\lambda_u^n\Pi_u+R_u^n\,  .
\end{equation}

We will make the following assumption on the expansion of $\lambda$ at $0$.
\begin{hypothesis}\label{HHH1}
Let $Y$ be a random variable with integrable characteristic
function $a_.:=e^{-\psi(\cdot)}$ and with density function $\Phi$.
Assume that there exists a sequence of invertible matrices $(\Theta_n)_n$ such that $\lim_{n\rightarrow +\infty}\Theta_n^{-1}=0$
and
\begin{equation}\label{asymptlambda}
\forall u,\quad \lambda_{^{t}\Theta_n^{-1}\cdot u}^n\sim e^{-\psi(u)}=a_u\, ,\quad\mbox{as}\ n\rightarrow +\infty
\end{equation}
(where $^{t}\Theta_n^{-1}$ stands for the transpose matrix of $\Theta_n^{-1}$) and
$$
\forall u\in[-b,b]^d,\quad 
   |\lambda_{u}^n|\le 2\left| e^{-\psi({}^t\Theta_n\cdot u)}\right| \, .$$
\end{hypothesis}
Note that, under Hypothesis \ref{HHH} and if \eqref{asymptlambda} holds true, then
$$\forall u\in\mathbb R^d,\quad 
e^{-\psi(u)}=\lim_{n\rightarrow +\infty}\lambda_{^{t}\Theta_n^{-1}\cdot u}^n=\lim_{n\rightarrow +\infty}\mathbb E_{\bar\mu}[P_{^{t}\Theta_n^{-1}\cdot u}^n\mathbf 1]=\lim_{n\rightarrow +\infty}\mathbb E_{\bar\mu}[e^{iu*(\Theta_n^{-1} S_n)}],$$
and so $(\Theta_n^{-1} S_n)_n$ converges in distribution to $Y$.
If
$Y$ has a stable distribution of index $\alpha\in(0,2]\setminus\{1\}$, i.e.
$$ \psi(u)=\int_{\mathbb S^1}|u*s|^{\alpha}(1+\tan\frac \pi\alpha \mbox{sign}(u*s))\, d\Gamma(u),$$
where $\Gamma$ is a Borel measure on the unit sphere $S^1=\{x\in\mathbb R^d\ :\ x*x=1\}$
and if
\begin{equation*}
\lambda_u
= e^{-\psi(u)L(|u|^{-1})} + o\left(|u|^\alpha L(|u|^{-1})\right)\, ,\quad \mbox{as }
   u\rightarrow 0\, ,
\end{equation*}
with $L$ slowly varying at infinity,
then Hypothesis \ref{HHH1} holds true with $\Theta_n:= \mathfrak a_n\, Id$ with
$\mathfrak{a}_n:= \inf\{x>0\, :\, n |x|^{-\alpha} L(x) \geq 1\}\, .
$

But Hypothesis \ref{HHH1} allows also the study of situations with anisotropic scaling.

Before stating our first general result, let us introduce an additional notation.
Under Hypothesis \ref{HHH}, for any function $u:\bar M\rightarrow\mathbb C$, we write $\Vert u\Vert_{\mathcal B'_0}:=\sup_{h\in\mathcal B_0}|\mathbb E_{\bar\mu}[u.h]|$.

\begin{theorem}\label{MAINGENE0}
Assume Hypotheses \ref{HHH} and \ref{HHH1}. Let $f,g:M\rightarrow \mathbb C$
be such that
$$\Vert f\Vert_{+}:=\sum_{\ell\in\mathbb Z^d} \Vert f(\cdot,\ell)\Vert
 <\infty\quad\mbox{and}\quad   \Vert g\Vert_{+,\mathcal B_0'}:=\sum_{\ell\in\mathbb Z^d}\Vert g(\cdot,\ell)\Vert_{\mathcal B_0 '}<\infty.$$
Then 
$$\int_Mf.g\circ T^n\, d\nu=\frac{\Phi(0)}{\det \Theta_n}\left(\int_Mf\, d\nu\, \int_Mg\, d\nu +o(1)\right),\ \mbox{as}\ n\rightarrow +\infty\, .$$
\end{theorem}
\begin{proof}
For every positive integer $n$ and every $\ell\in\mathbb Z^d$,
combining \eqref{formulespectrale} with Hypothesis \ref{HHH},
the following equalities hold in $\mathcal L(\mathcal B,\mathcal B_0)$:
\begin{eqnarray}
P^n(\mathbf 1_{S_n=\ell}\cdot)
  &=&\frac 1{(2\pi)^d}\int_{[-b,b]^d}
  e^{-it*\ell}\lambda_t^n\Pi_t(\cdot)\, dt+O(\vartheta^n)\nonumber\\
  &=&\frac 1{(2\pi)^d\det{\Theta_n}}\int_{{}^t\Theta_n[-b,b]^d}
  e^{-iu*(\Theta_ n^{-1}\ell)}\lambda_{{}^t\Theta_n^{-1} u}
       ^n\Pi_{{}^t\Theta_n^{-1} u}(\cdot)\, du+O(\vartheta^n)\nonumber\\
  &=&\frac 1{(2\pi)^d\det{\Theta_n}}\int_{\mathbb R^d}
  e^{-iu*(\Theta_ n^{-1}\ell)}e^{-\psi(u)}
      \Pi_{0}(\cdot)\, du+\varepsilon_{n,\ell}\nonumber\\
&=&\frac{\Phi(\Theta_n^{-1} \ell)}{\det{\Theta_n}}\Pi_0 +\varepsilon_{n,\ell}\, ,\label{egalitecentrale}
\end{eqnarray}
with $\sup_\ell \Vert\varepsilon_{n,\ell}\Vert_{\mathcal L(\mathcal B,\mathcal B_0)}=o(\det\Theta_n^{-1})$
due to the dominated convergence theorem applied to 
$\left\Vert\lambda^n_{{}^t\Theta_n^{-1}u}\Pi_{{}^t\Theta_n^{-1}u}-e^{-\psi(u)}\Pi_0\right\Vert_{\mathcal L(\mathcal B,\mathcal B_0)}\mathbf 1_{{}^t\Theta_n[-b,b]^d}(u)$.
Setting $u_\ell:=f(\cdot,\ell)$ and $v_\ell:=g(\cdot,\ell)$ and using
\eqref{FORMULECLEF}, we obtain
\begin{eqnarray}
\int_Mf.g\circ T^n\, d\nu
&=&\sum_{\ell,\ell'\in\mathbb Z^d}\left(\frac{\Phi(\Theta_n^{-1}( {\ell'-\ell}))}{\det\Theta_n}\mathbb E_{\bar\mu}[u_\ell]\, \mathbb E_{\bar\mu}[v_{\ell'}]+\mathbb E_{\bar\mu}[v_{\ell'}\varepsilon_{n,\ell}(u_\ell)]\right)\nonumber\\
&=&\sum_{\ell,\ell'\in\mathbb Z^d}\left(\frac{\Phi(\Theta_n^{-1}( {\ell'-\ell}))}{\det\Theta_n}\mathbb E_{\bar\mu}[u_\ell]\, \mathbb E_{\bar\mu}[v_{\ell'}]\right)+O\left(\sum_{\ell,\ell'\in\mathbb Z^d}\Vert v_{\ell'}\Vert_{\mathcal B_0'}\, \Vert\varepsilon_{n,\ell}\Vert_{\mathcal L(\mathcal B,\mathcal B_0)}\Vert u_\ell\Vert\right)\, \nonumber\\
&=&\sum_{\ell,\ell'\in\mathbb Z^d}\frac{\Phi(\Theta_n^{-1} (\ell'-\ell))}{\det\Theta_n}\mathbb E_{\bar\mu}[u_\ell]\, \mathbb E_{\bar\mu}[v_{\ell'}] +\tilde\varepsilon_n(f,g)\, ,\label{controlecentral}
\end{eqnarray}
with $ \lim_{n\rightarrow +\infty}\sup_{f,g}\frac{\det\Theta_{n}\, \tilde\varepsilon_n(f,g)}{\Vert g\Vert_{+,\mathcal B_0'}\Vert f\Vert_{+}}=0$.
Now, due to the dominated convergence theorem and since $\Phi$
is continuous and bounded,
$$ \lim_{n\rightarrow +\infty}\sum_{\ell,\ell'\in\mathbb Z^d}{\Phi\left(\Theta_n^{-1}(\ell'-\ell)\right)}\mathbb E_{\bar\mu}[u_\ell]\, \mathbb E_{\bar\mu}[v_{\ell'}]=\Phi(0)\sum_{\ell,\ell'\in\mathbb Z^2}\mathbb E_{\bar\mu}[u_\ell]\, \mathbb E_{\bar\mu}[v_{\ell'}]=
\Phi(0)\int_Mf\, d\nu\, \int_Mg\, d\nu\, ,$$
which ends the proof.
\end{proof}
We will reinforce Hypothesis \ref{HHH1}. Notations $\lambda_0^{(k)}$, $a_0^{(k)}$,
$\Pi_0^{(k)}$ stand for the $k$-th derivatives of $\lambda$, $a$ and $\Pi$
at 0.
\begin{theorem}\label{THMGENE}
Assume Hypothesis \ref{HHH} with $\mathcal B_0=\mathcal B$.
Let $K, M,P$ be three integers such that $K\ge d/2$, $ 3\le P\le M+1$ and
\begin{equation}\label{MAJOM}
-\left\lfloor \frac M P\right\rfloor+ \frac{M}2\ge K\, .
\end{equation}
Assume moreover that $\lambda_\cdot$
is $C^{M}$-smooth and that there exists a positive symmetric matrix $\Sigma^2$
such that
\begin{equation}\label{DLlambda}
\lambda_u -1\sim -\psi(u):=-\frac 12\Sigma^2*u^{\otimes 2}\, ,\quad\mbox{as}\ u\rightarrow 0\, .
\end{equation}
Assume that, for every $k<P$, $\lambda^{(k)}_0=a^{(k)}_0$ with $a_t=e^{-\psi(t)}$, for every $k<P$.
Assume moreover that the functions
$\Pi$ and $R$ are $C^{2K}$-smooth.
Let $f,g:M\rightarrow\mathbb C$ be such that
\begin{equation}\label{hypofg}
\sum_{\ell\in\mathbb Z^d} (\Vert f(\cdot,\ell)\Vert+
    \Vert g(\cdot,\ell)\Vert_{\mathcal B '})<\infty\, .
    \end{equation}
Then
\begin{equation}\label{FFFF1}
\int_Mf.g\circ T^n\, d\nu=\sum_{\ell,\ell'\in\mathbb Z^d} \sum_{m=0}^{2K}\frac {1}{m!}\sum_{j=0}^{M}
\frac {i^{m+j}}{(j)!}\frac{\Phi^{(m+j)}\left(\frac {\ell'-\ell}{\mathfrak a_{n}}\right)}{n^{\frac{d+m+j}2}}*(\mathbb E_{\bar\mu}[v_{\ell'} \Pi^{(m)}_0(u_\ell)]\otimes(\lambda^n/a^n)_0^{(j)})+o(n^{-K-\frac d2})\, .
\end{equation}
If moreover $\sum_{\ell\in\mathbb Z^d}|\ell|^{2K}(\Vert f(\cdot,\ell)\Vert+\Vert g(\cdot,\ell)\Vert_{\mathcal B'})<\infty$, then
\begin{eqnarray}
\int_Mf.g\circ T^n\, d\nu
&=&\sum_{m,j,r}\frac {i^{j+m}}{m!\, r!\, j!}
\left(\frac{\Phi^{(j+m+r)}(0)}{n^{\frac{j+d+m+r}2}}*(\lambda^n/a^n)_0^{(j)}\right)\nonumber\\
&\ &*\sum_{\ell,\ell'\in\mathbb Z^d}(\ell'-\ell)^{\otimes r}\otimes \mathbb E_{\bar\mu}[v_{\ell'} \Pi^{(m)}_0(u_\ell)]
+o(n^{-K-\frac d2})\, ,\label{decorrelation2b}
\end{eqnarray}
where the sum is taken over the $(m,j,r)$ with $m,j,r$ non negative integers such that $j+m+r\in 2\mathbb Z$ and $\frac{r+m+j}2-\lfloor \frac j P\rfloor\le K$.
\end{theorem}
Observe that
$$ (\lambda^n/a^n)^{(j)}_0=\sum_{k_1m_1+...+k_rm_r=j}\frac{n!}{m_1!\cdots m_r!(n-m_1-...-m_r)!}((\lambda/a)_0^{(k_1)})^{m_1}\cdots ((\lambda/a)_0^{(k_r)})^{m_r}\, ,$$
where the sum is taken over $r\ge 1$, $m_1,...,m_r\ge 1$, $k_r>...>k_1\ge P$ (this implies that $m_1+...+m_r\le j/P$). 
Hence $(\lambda^n/a^n)^{(j)}_0$
is polynomial in $n$ with degree at most $\lfloor j/P\rfloor$.
\begin{remark}
Note that \eqref{MAJOM} holds true as soon as $M\ge 2KP/(P-2)$
and $M$ in \eqref{FFFF1} can be replaced by $(2K-m)P/(P-2)$.

Moreover \eqref{decorrelation2b} provides an expansion of the following form:
$$\int_Mf.g\circ T^n\, d\nu=\sum_{m=0}^K\frac{c_m(f,g)}{n^{\frac d2+m}}+o(n^{-K-\frac d2}) \, .
$$
\end{remark}
\begin{remark}\label{DEVASYMP}
If $\Pi$ is $C^M$-smooth, using the fact $(\lambda^n/a^n)^{(j)}_0=O(n^{\lfloor j/P\rfloor})$, if $\sum_{\ell\in\mathbb Z^d}|\ell|^{M}(\Vert f(\cdot,\ell)\Vert+\Vert g(\cdot,\ell)\Vert_{\mathcal B'})<\infty$ the right hand side of \eqref{decorrelation2b} can be rewritten
\[
\frac 1{n^{\frac d2}}
\sum_{\ell,\ell'\in\mathbb Z^d}\sum_{L=0}^{M}\frac 1{n^{L/2}}\frac{\Phi^{(L)}(0)}{L!}
i^L\frac{\partial^L}{\partial t^L}\left(\mathbb E_{\bar\mu}\left[v_{\ell'}. e^{-it*(\ell'-\ell)}.\lambda_{t}^n\Pi_{t}.u_{\ell}\right]e^{\frac{n}2\Sigma^2*t^{\otimes 2}}\right)_{|t=0}
+o(n^{-K-\frac d2})\, .
\]
If moreover $\sup_{u\in [-b,b]^d} \Vert{(R_u^n)^{(m)}}\Vert_{\mathcal(\mathcal B,\mathcal B)} =O( \vartheta^n)$
 for every $m=0,...,M$, then it can also be rewritten
\[
\frac 1{n^{\frac d2}}\sum_{\ell,\ell'\in\mathbb Z^d}\sum_{L=0}^{M}\frac{\Phi^{(L)}(0)}{L!}i^L
\frac{\partial^L}{\partial t^L}\left(\mathbb E_{\bar\mu}\left[u_\ell.e^{it*\frac{S_n-(\ell'-\ell)}{\sqrt n}}.v_{\ell'}\circ\bar T^n\right]e^{\frac{1}2\Sigma^2*t^{\otimes 2}}\right)_{|t=0}
+o(n^{-K-\frac d2})\, ,
\]
where we used \eqref{decomp2}.
\end{remark}
\begin{proof}[Proof of Theorem \ref{THMGENE}]
We assume, up to a change of $b$ that Hypothesis \ref{HHH1} holds true.
Due to \eqref{formulespectrale} and to \eqref{decomp2}, in $\mathcal L(\mathcal B,\mathcal B)$, we have
\begin{eqnarray*}
P^n(\mathbf 1_{S_n=\ell}\cdot)
&=&\frac 1{(2\pi)^d}\int_{[-\pi,\pi]^d}e^{-it*\ell}P_t^n(\cdot)\, dt\\
&=&\frac 1{(2\pi)^d}\int_{[-b,b]^d}e^{-it*\ell}\lambda_t^{n}\Pi_t(\cdot)\, dt
+O(\vartheta^{n})\\
&=&\frac 1{(2\pi)^dn^{\frac{d}2}}\int_{[-b\sqrt{n},b\sqrt{n}]^d}e^{-it*\frac{\ell}{\sqrt{n}}}\lambda_{t/\sqrt{n}}^{n}\Pi_{t/\sqrt{n}}(\cdot)\, dt
+O(\vartheta^{n})\\
&=&\frac 1{(2\pi)^dn^{\frac{d}2}}\int_{[-b\sqrt{n},b\sqrt{n}]^d}e^{-it*\frac{\ell}{\sqrt{n}}}\lambda_{t/\sqrt{n}}^{n} \sum_{m=0}^{2K}\frac 1{m!}\Pi^{(m)}_0(\cdot)*\frac{t^{\otimes m}}{n^{\frac m2}}\, dt
+o(n^{-K-\frac d2})\, ,
\end{eqnarray*}
due to the dominated convergence theorem since there exists $x_{t/\sqrt{n}}\in(0,t/\sqrt{n})$ such that
$\Pi_{t/\sqrt{n}}(\cdot)= \sum_{m=0}^{2K-1}\frac 1{m!}\Pi^{(m)}_0(\cdot)*\frac{t^{\otimes m}}{n^{\frac m2}}+\frac 1{(2K)!}\Pi^{(2K)}_0(x_{t/\sqrt{n}})*\frac{t^{\otimes 2K}}{n^{K}}$.
Recall that $(\lambda^n/a^n)_0^{(j)}=O(n^{\lfloor j/P\rfloor})$, so
\[
\left\vert \lambda_{t/\sqrt{n}}^{n}
   - a_t\sum_{j=0}^{M}\frac 1{j!}(\lambda^n/a^n)_0^{(j)}*\frac{t^{\otimes j}}{n^{\frac j2}} 
       \right\vert\\
\le  n^{\lfloor \frac M P\rfloor}a_t \frac{|t|^{M}}{n^{\frac M2}} \eta(t/\sqrt{n})\, ,
\]
with $\lim_{t\rightarrow 0}\eta(t)=0$ and $\sup_{[-b,b]^d}|\eta|<\infty$.
Due to \eqref{MAJOM}, we obtain
\begin{eqnarray*}
P^n(\mathbf 1_{S_n=\ell}\cdot)&=&\frac 1{(2\pi)^dn^{\frac d2}}\int_{[-b\sqrt{n},b\sqrt{n}]^d}e^{-it*\frac{\ell}{\sqrt{n}}}e^{-\frac 12\Sigma^2* t^{\otimes 2}}
 \sum_{m=0}^{2K}\frac 1{m!}\Pi^{(m)}_0(\cdot)*\frac{t^{\otimes m}}{n^{\frac m2}}\\
&\ &\left(1+\sum_{j=P}^{M}\frac 1{j!}   (\lambda^n/a^n)_0^{(j)}*\frac{t^{\otimes j}}{n^{\frac j2}}\right)\, dt
+ o\left({n^{-K-\frac d2}}\right)\\
&=&\sum_{m=0}^{2K}\sum_{j=0}^{M}\frac{i^{m+j}}{n^{\frac {m+j+d}2}\, m!\, j!}
\Phi^{(m+j)}\left(\frac \ell {\sqrt{n}}\right)*\left(\Pi^{(m)}_0(\cdot)\otimes (\lambda^n/a^n)_0^{(j)}\right)+o(n^{-K-\frac d2}).\label{INT22}
\end{eqnarray*}
This combined with \eqref{FORMULECLEF} and \eqref{hypofg} gives \eqref{FFFF1}.

We assume from now on that 
$\sum_{\ell\in\mathbb Z^d}|\ell|^{2K}(\Vert f(\cdot,\ell)\Vert+\Vert g(\cdot,\ell)\Vert_{\mathcal B'})$. Recall that $(\lambda^n/a^n)^{(j)}_0$ is polynomial in $n$ of degree at most $\lfloor j/P\rfloor$.
Hence, due to the dominated convergence theorem, we can replace
$\Phi^{(m+j)}\left(\frac {\ell'-\ell} {\sqrt{n}}\right)$ in \eqref{FFFF1}
by
$$\sum_{r=0}^{2K-m-j+2\lfloor\frac{j}P\rfloor}\frac 1{r!\, n^{\frac r 2}}\Phi^{(m+j+r)}(0)*(\ell'-\ell)^{\otimes r}\, .$$
Hence we have proved \eqref{decorrelation2b}.
\end{proof}
Now, we come back to the case of $\mathbb Z^2$-periodic Sinai billiards, with the notations of Section \ref{sec:model}.
\section{Young towers for billiards}\label{sec:young}
Recall that, in \cite{young}, Young constructed two dynamical systems $(\tilde M,\tilde T, \tilde\mu)$ and $(\hat M, \hat T, \hat \mu)$ and  two measurable functions $\tilde\pi\colon \tilde M\to \bar M$ and $\hat \pi\colon \tilde M\to \hat M$ such that 
$$\tilde\pi\circ \tilde T=\bar T\circ \tilde\pi,\  \tilde\pi_*\tilde\mu=\bar\mu,\ \hat\pi\circ \tilde T=\hat T\circ \hat\pi,\  \hat\pi_*\tilde\mu=\hat\mu$$
and such that, for every measurable $f\colon \bar M\to\CC$ constant on every stable manifold, there exists $\hat f\colon \hat M\to\CC$ such that $\hat f\circ \hat\pi=f\circ \tilde\pi$.
We consider the partition $\hat{\D}$ on $\hat M$ constructed by Young
in \cite{young} together with the separation time given, for every $x,y$, by
\[
s_0(x,y):=\min\{n\ge -1:\ \hat{\D}(\hat T^{n+1}x)\ne \hat {\D}(\hat T^{n+1}y)\}.
\]
It will be worth noting that, for any $x,y$, the sets $\tilde\pi\hat\pi^{-1}\{x\}$ and $\tilde\pi\hat\pi^{-1}\{y\}$
are contained in the same connected component of $\bar M\setminus\bigcup_{k=0}^{s_0(x,y)}\bar T^{-k}\mathcal S_0$.

Let $p>1$ and set $q$ such that $\frac1p+\frac1q=1$. Let $\eps>0$ and $\beta\in(0,1)$ be suitably chosen and let us define
\[
\|\hat f\| = \sup_{\ell} \|\hat f_{|\hat \Delta_\ell}\|_\infty e^{-\ell\eps}
+
\sup_{\hat A\in\hat{\mathcal D}} \esssup_{x,y\in \hat A} \frac{|\hat f(x)-\hat f(y)|}{\beta^{s_0(x,y)}}e^{-\ell \eps}.
\]
Let $\B:=\{\hat f\in L^q_\CC(\hat M,\hat\mu)\colon \|\hat f\|<\infty\}$.
Young proved that the Banach space $(\B,\|\cdot\|)$ satisfies $\|\cdot\|_{q}\le \|\cdot\|$, that the transfer opertor $\hat P$ on $\mathcal B$
($\hat P$ being defined on $L^q$ as the adjoint of the composition by $\hat T$ on $L^p$) is quasicompact on $\B$. We assume without any loss of generality (up to an adaptation of the construction of the tower) that the dominating eigenvalue of $\hat P$ on $\B$ is $1$ and is simple.

Since $\kappa\colon \bar M\to\ZZ^2$ is constant on the stable manifolds, there exists $\hat \kappa\colon\hat M\to\ZZ^2$ such that $\hat \kappa\circ \hat \pi=\kappa\circ\tilde\pi$. We set $\hat S_n:=\sum_{k=0}^{n-1}\hat\kappa\circ\hat T^k$.
For any $u\in\RR^2$ and $\hat f\in\mathcal B$, we set $\hat P_u(\hat f):=\hat P(e^{i u*\hat\kappa}\hat f)$.

\begin{proposition}\label{pro:pertu2}
$t\mapsto \lambda_t$ is an even function.
\end{proposition}
\begin{proof}
Let $\Psi:\bar M\rightarrow \bar M$ be the map which sends
$(q,\vec v)\in\bar M$  to $(q,\vec v')\in\bar M$ 
such that $\widehat{(\vec n(q),\vec v')}=-\widehat{(\vec n(q),\vec v)}$.
Then $\kappa\circ \bar T^k\circ\Psi=-\kappa\circ \bar T^{-k-1}$.
Hence, $S_n$ as the same distribution (with respect to $\bar\mu$) as $-S_n$ and so
$$\forall t\in[-b,b]^2,\quad \mathbb E_\mu[e^{-it*S_n}]=\mathbb E_\mu[e^{it*S_n}]\sim \lambda_t^n \mathbb E_{\hat\mu}[\Pi_t \mathbf 1]\sim \lambda_{-t}^n \mathbb E_{\hat\mu}[\Pi_{-t} \mathbf 1]$$
as $n$ goes to infinity, and so $\lambda$ is even.
\end{proof}
Let $\mathcal Z_{k}^m$ be the partition of $\bar M\setminus \bigcup_{j=k}^m\bar T^{-j}(\mathcal S_0)$ into its connected components.
We also write $\mathcal Z_k^{\infty}:=\bigvee _{j\ge k}\mathcal Z_k^j$.
\begin{proposition}\label{pro:pertu2b}
Let $k$ be a nonnegative integer and let $u,v:\bar M\rightarrow\mathbb C$ be respectively $\mathcal Z_{-k}^k$-measurable and $Z_{-k}^\infty$-measurable functions.

Then there exists $\hat u,\hat v:\hat M\rightarrow \mathbb C$
such that $u\circ\bar T^k\circ\tilde\pi=\hat u\circ\hat \pi$ and $v\circ\bar T^k\circ\tilde\pi=\hat v\circ\hat \pi$.

Moreover, $\hat u\in\mathcal B$ and for every $t\in\mathbb R$, $ \hat P_t^{2k}(e^{-it*\hat S_k}\hat u)=\hat P^{2k}(e^{it*\hat S_k\circ\hat T^k}\hat u)$
and
\begin{equation}\label{P2ku}
\Vert \hat P^{2k}(e^{it*\hat S_k\circ\hat T^k}\hat u)\Vert\le (1+2\beta^{-1})\Vert u\Vert_\infty\, ,
\end{equation}
and
\begin{equation}\label{EqClef}
\forall n>k,\quad
\mathbb E_{\bar\mu}[u.e^{it*S_n}.v\circ\bar T^n]=
\mathbb E_{\hat\mu}[\hat v.e^{it*\hat S_k}\hat P_t^n(e^{-it*\hat S_k}\hat u)]\, .
\end{equation}
\end{proposition}
\begin{proof}
Using several times $\hat P^m(f.g\circ T^m)=g.\hat P^mf$ and $\hat P_t^m=\hat P^m(e^{it\hat S_m}\cdot)$, we obtain
\begin{eqnarray*}
\mathbb E_{\bar\mu}[u.e^{it*S_n}.v\circ\bar T^n]&=&
   \mathbb E_{\bar\mu}[u\circ \bar T^k.e^{it*S_n}\circ \bar T^k.v\circ\bar T^{n+k}]\\
&=&\mathbb E_{\hat\mu}[\hat u. e^{it* \hat S_n}\circ \hat T^k.\hat v\circ\hat T^n]\\
&=&\mathbb E_{\hat\mu}[\hat P^{n+k}(\hat u. e^{it* (\hat S_{n-k}\circ \hat T^k +\hat S_k\circ \hat T^n)}.\hat v\circ\hat T^n)]\\
&=&\mathbb E_{\hat\mu}[\hat P^{k}(e^{it*\hat S_k}\hat v.\hat P^{n}( e^{it* (\hat S_{n-k}\circ \hat T^k)}.\hat u))]\\
&=&\mathbb E_{\hat\mu}[\hat P^{k}(e^{it*\hat S_k}\hat v.\hat P_t^{n}(e^{-it*\hat S_k}\hat u))]\, ,
\end{eqnarray*}
since $\hat S_{n-k}\circ\hat T^k=\hat S_n-\hat S_k$.
Hence, we have proved \eqref{EqClef} (since $\hat P$ preserves $\hat\mu$).
\end{proof}

\section{Proofs of our main results in the finite horizon case}\label{sec:MAIN}
We assume throughout this section that the billiard has finite horizon.

The Nagaev-Guivarc'h method \cite{nag1,nag2,GH}
has been applied in this context by Sz\'asz and Varj\'u \cite{SV1} (see also \cite{FP09a}) to prove Hypotheses \ref{HHH} and 
\ref{HHH1} hold for $\mathcal B_0=\mathcal B$ the Young Banach space.
More precisely, we have the following.
\begin{proposition}[\cite{SV1,FP09a}]\label{pro:pertu}
There exist a real $b\in(0,\pi)$ and three $C^\infty$ functions $t\mapsto \lambda_t$, $t\mapsto \Pi_t$ and $t\mapsto N_t$ defined on $[-b,b]^2$ and with values in $\mathbb C$, $\mathcal L(\mathcal B,\mathcal B)$ and $\mathcal L(\mathcal B,\mathcal B)$ respectively such that
\begin{itemize}
\item[(i)] for every
$ t\in[-b,b]^2$,  
$\hat P_t^n=\lambda_t^n \Pi_t + N_t^n$and $\Pi_0=\mathbb E_{\hat\mu}[\cdot]$, $\Pi_t\hat P_t=\hat P_t\Pi_t=\lambda_t\Pi_t$, $\Pi_t^2=\Pi_t$;
\item[(ii)]
there exists $\vartheta\in(0,1)$ such that, for every positive integer $m$,
\[
\sup_{t\in[-b,b]^2} \| (N^n)^{(m)}_t\|_{\mathcal L(\mathcal B,\mathcal B)} = O(\vartheta^n)  
\quad\text{and}\quad
\sup_{t\in[-\pi,\pi]^2\setminus[-b,b]^2} \| \hat P_t^n\|_{\mathcal L(\mathcal B,\mathcal B)} = O(\vartheta^n);
\]
\item[(iii)]
we have $\lambda_t = 1-\frac12\Sigma^2*t^{\otimes 2} = O(|t|^3)$;
\item[(iv)] there exists $\sigma>0$ such that,
for any $t\in[-b,b]^2$, 
$|\lambda_t|\le e^{-\sigma|t|^2} $ and
$e^{-\frac 12\Sigma^2*t^{\otimes 2}}\le e^{-\sigma|t|^2} $.
\end{itemize}
\end{proposition}


Our first step consists in stating a high order expansion of the following quantity
$$\mathbb E_{\bar\mu}[u.\mathbf 1_{S_n=\ell}.v\circ\bar T^n]$$
for $u$ and $v$ dynamically Lispchitz on $\bar M$.
Let us recall that, due to \eqref{FORMULECLEF0}, this result corresponds to a mixing result for observables supported on a single cell.
We start by studying this quantity
for some locally constant observables. This result is a refinement of
\cite[prop. 4.1]{ps10} (see also \cite[prop 3.1]{FP17}.
Let $\Phi$ be the density function of $B$, which is given by
$\Phi(x)=\frac {e^{-\frac {(\Sigma^2)^{-1}*x^{\otimes 2}}2}}{2\pi\, \sqrt{\det\Sigma^2}}$.
\subsection{A first local limit theorem}
We set $a_t:= e^{-\frac{1}2\Sigma^2*t^{\otimes 2}}$.
Note that the uneven derivatives of $\lambda/a$ at 0 are null
as well as its three first derivatives.
\begin{proposition}\label{TLL}
Let $K$ be a positive integer and a real number $p>1$. There exists $c>0$ such that, 
for any $k\ge 1$, if $u,v:\bar M\rightarrow \mathbb C$ are respectively $\Z_{-k}^{k}$-measurable and $\Z_{-k}^\infty$-measurable, 
then for any $n > 3k$ and $\ell\in\ZZ^2$
\begin{eqnarray}
&\ &\left|
\mathbb E_{\bar\mu}\left[u\mathbf 1_{\{S_n=\ell\}}.v\circ \bar T^{n}\right]
-\sum_{m=0}^{2K-2}\frac {1}{m!}\sum_{j=0}^{2K-2-m}
\frac {i^{m+2j}}{(2j)!}\frac{\Phi^{(m+2j)}\left(\frac \ell{\sqrt{n}}\right)}{n^{j+1+\frac{m}2}}*(A_m(u,v)\otimes(\lambda^n/a^n)_0^{(2j)})
\right|\nonumber\\
&\ &
\le \frac{ck^{2K-1}\Vert v\Vert_p\, \Vert u\Vert_\infty}{n^{K+\frac14}}\, , \label{formuleTLL}
\end{eqnarray}
with, for every $m\in\{0,...,4K-4\}$,
\begin{equation}\label{Am}
\left| A_m(u,v)-
\frac{\partial^m}{\partial t^m}\left(\frac{\mathbb E_{\bar\mu}[u.e^{it*S_n}.v\circ \bar T^n]}{\lambda_t^n}\right)_{|t=0}
    \right|\le cn^m\vartheta^{n-2k}\Vert v\Vert_p\Vert u\Vert_\infty\, ,
\end{equation}
\begin{equation}\label{MAJO}
\left| A_m(u,v)\right|\le c\, k^m\Vert v\Vert_p\Vert u\Vert_\infty\quad\mbox{and}\quad(\lambda^n/a^n)_0^{(m)}=O(n^{m/4}).
\end{equation}
In particular, for $K=2$, we obtain
\begin{eqnarray*}
&\ &\left|
\mathbb E_{\bar\mu}\left[u\mathbf 1_{\{S_n=\ell\}}.v\circ \bar T^{n}\right]
-\frac{\Phi(\frac\ell{\sqrt{n}})}{n}A_0(u,v)
- \frac i{n^{\frac 32}}\Phi'(\frac\ell{\sqrt{n}}) * A_1(u,v)\right.\\
&\ &\left. + \frac 1{n^2}\Phi"(\frac \ell{\sqrt{n}}) * A_2(u,v)
-\frac {A_0(u,v)}{ n^2}
.\Phi^{(4)}\left(\frac \ell{\sqrt{n}}\right)*(\lambda^n/a^n)^{(4)}_0\right|\\
&\ &
\le \frac{ck^3\Vert v\Vert_p\, \Vert u\Vert_\infty}{n^\frac94}\, .
\end{eqnarray*}
\end{proposition}
\begin{remark}\label{FORMULETLL2}
Due to \eqref{Am} and \eqref{MAJO}, \eqref{formuleTLL} can be rewritten
as follows:
\begin{eqnarray*}
&\ &\left|
\mathbb E_{\bar\mu}\left[u\mathbf 1_{\{S_n=\ell\}}.v\circ \bar T^{n}\right]
-\sum_{m=0}^{4K-4}\frac {i^m}{m!} \frac{\Phi^{(m)}(\frac\ell{\sqrt{n}})}{n^{1+\frac m2}}*\left(e^{\frac n2\Sigma^2*t^{\otimes 2}}\mathbb E_{\bar\mu}[u.e^{it*S_n}.v\circ \bar T^n]\right)^{(m)}_{|t=0}\right|\\
&\ &
\le \frac{ck^{4K-4}\Vert v\Vert_p\, \Vert u\Vert_\infty}{n^{K+\frac14}}\, .
\end{eqnarray*}
\end{remark}
\begin{proof}[Proof of Proposition \ref{TLL}]
Since 
$u\circ \bar T^k$ is $\Z_{0}^{2k}$-measurable and $v\circ\bar T^{k}$ is $\Z_0^\infty$-measurable, there exist $\hat u,\hat v:\hat M\rightarrow\mathbb C$ such that
$\hat u\circ\hat\pi = u\circ\bar T^k\circ \tilde\pi$ 
and $\hat v\circ\hat\pi = v\circ\bar T^k\circ\tilde\pi$, with $\hat u\in\mathcal B$.
As in the proof of \cite[Prop. 4.1]{ps10}, we set 
\[
C_n(u,v,\ell):=\mathbb E_{\bar\mu}[u.\mathbf 1_{\{S_n=\ell\}}. v\circ \bar T^n]\, .
\]
Due to \eqref{EqClef}, we obtain
\begin{eqnarray}
C_n(u,v,\ell)&=&\frac 1{(2\pi)^2}\int_{[-\pi,\pi]^2}e^{-it*\ell}\mathbb E_{\bar\mu}[u. e^{it *S_n}.v\circ\bar T^n]\, dt\nonumber\\
&=&\frac 1{(2\pi)^2}\int_{[-\pi,\pi]^2}e^{-it*\ell}\mathbb E_{\hat\mu}[e^{it *\hat S_k}\hat v.\hat P_t^{n}(e^{-it*\hat S_k}\hat u)]\, dt
\, .
\end{eqnarray}
Let $\Xi_{k,t}:=e^{it*\hat S_k}\Pi_t(e^{-it*\hat S_k}\cdot)$.
We will write $\Xi^{(m)}_{k,0}$ for $\frac{\partial^m}{\partial t^m}(\Xi_{k,t})_{|t=0}$.
Due to items (i) and (ii) of
Proposition~\ref{pro:pertu} and due to \eqref{P2ku}, it comes
\begin{eqnarray}
C_n(u,v,\ell)&=&\frac 1{(2\pi)^2}\int_{[-b,b]^2}e^{-it*\ell}\lambda_t^{n-2k}
\mathbb E_{\hat\mu}[e^{it *\hat S_k}\hat v.\Pi_t\hat P_t^{2k}(e^{-it*\hat S_k}\hat u)]\, dt+O(\vartheta^{n-2k}\Vert  u\Vert_\infty.\Vert v\Vert_p )\nonumber\\
&=&\frac 1{(2\pi)^2}\int_{[-b,b]^2}e^{-it*\ell}\lambda_t^{n}\mathbb E_{\hat\mu}[\hat v.{\Xi}_{k,t}\hat u]\, dt+O(\vartheta^{n-2k}\Vert u\Vert_\infty.\Vert v\Vert_p )\, ,
\end{eqnarray}
since $\Pi_t \hat P_t=\lambda_t\Pi_t$ and $\Pi_t^2=\Pi_t$ so that 
\begin{equation}\label{formuleXi}
\Xi_{k,t}=\lambda_t^{-2k}e^{it*\hat S_k}\Pi_t\hat P_t^{2k}(e^{-it*\hat S_k}\cdot)\, .
\end{equation}
Observe that
\begin{equation}\label{MAJOO}
\frac{1}{(2\pi)^2}\int_{[-b,b]^2}|t|^{j}|\lambda_t|^{n}\, dt\le 
\frac{1}{(2\pi)^2 n^{\frac {j+2}2}}\int_{[-b\sqrt{n},b\sqrt{n}]^2}|t|^{j}e^{-\sigma |t|^2}\, dt\, ,
\end{equation}
 and so
\begin{equation}
C_n(u,v,\ell)=\frac 1{(2\pi)^2}\int_{[-b,b]^2}e^{-it*\ell}\lambda_t^{n} \sum_{m=0}^{2K-2}\frac 1{m!}A_m(u,v)*t^{\otimes m}\, dt+ O\left(\frac{k^{2K-1}\Vert v\Vert_p\, \Vert  u\Vert_\infty}
{n^{K+\frac 12}}\right)\, ,
\end{equation}
with $A_m(u,v):=\mathbb E_{\hat\mu}[\hat v.\Xi^{(m)}_{k,0}\hat u]$.
Indeed $\Xi^{(2K-1)}_{k,s}\hat u$
is a linear combination of terms of the form
$$e^{is*\hat S_k}.(i\hat S_k)^{\otimes a}\otimes \Pi^{(b)}_s\hat P^{2k}(\otimes (i\hat S_k\circ\hat T^k)^{\otimes c}e^{is*\hat S_k\circ\hat T^k}\hat u)\otimes(\lambda^{-2k})^{(d)}_s$$
over nonnegative integers $a,b,c,d$ such that $a+b+c+d=2K-1$,
and these terms are in $O(k^{2K-1}\Vert u\Vert_\infty)$ in $\mathcal B$, uniformly in $k$.
Moreover, due to \eqref{formuleXi}, to \eqref{EqClef} and to Item (i) of Proposition \ref{pro:pertu},  we obtain
\begin{eqnarray*}
\forall t\in[-b,b]^2,\quad
\mathbb E_{\hat\mu}[\hat v.\Xi_{k,t}.\hat u]&=&\frac{\lambda_t^{n-2k}\mathbb E_{\hat\nu}[\hat v.e^{it*\hat S_k}\Pi_t\hat P_t^{2k}(e^{-it*\hat S_k}\hat u)]}{\lambda_t^n}\\
&=&
\frac{\mathbb E_{\bar\mu}[u.e^{it* S_n}. v\circ\bar T^n]-\mathbb E_{\hat\mu}[e^{it*\hat S_k}\hat v.N_{t}^{n-2k}\hat P_t^{2k}(e^{-it*\hat S_k}\hat u)]}{\lambda_t^n}\, ,
\end{eqnarray*}
so that
\[
\mathbb E_{\hat\mu}[\hat v.\Xi_{k,0}^{(m)}.\hat u]
=\left(\frac{\mathbb E_{\bar\mu}[u.e^{it*S_n}.v\circ \bar T^n]}{\lambda_t^n}\right)^{(m)}_{|t=0}
+O\left(n^m\vartheta^{n-2k}\Vert v\Vert_p\Vert u\Vert_\infty\right)\, .
\]
Recall that $a_t= e^{-\frac{1}2\Sigma^2*t^{\otimes 2}}$.
Since the three first derivatives of $\lambda$ and $a$ coincide, we have $(\lambda^n/a^n)_0^{(j)}=O(n^{j/4})$ and
\[
\left\vert \lambda_t^{n}
   - a_t^{n}\sum_{j=0}^{4K-4-2m}\frac 1{j!}(\lambda^n/a^n)_0^{(j)}*t^{\otimes j} 
       \right\vert\\
\le c_K n^{\frac{4K-3-2m}4}a_t^n |t|^{4K-3-2m}.
\]
Due to the analogue of \eqref{MAJOO} with $\lambda_t$ replaced by $a_t$, we obtain
\begin{eqnarray*}C_n(u,v,\ell)&=&\frac 1{(2\pi)^2}\int_{[-b,b]^2}e^{-it*\ell}e^{-\frac{n}2\Sigma^2* t^{\otimes 2}}
 \sum_{m=0}^{2K-2}\frac 1{m!}A_m(u,v)*t^{\otimes m}\\
&\ &\left(1+\sum_{j=4}^{4K-4-2m}\frac 1{j!}   (\lambda^n/a^n)_0^{(j)}*t^{\otimes j}\right)\, dt
+ O\left(\frac{k^{2K-1}\Vert v\Vert_p\, \Vert u\Vert_\infty}
{n^{K+\frac14}}\right).
\end{eqnarray*}

Note that
\begin{eqnarray}
&\ &\frac 1{(2\pi)^2}\int_{[-b,b]^2}e^{-it*\ell}e^{-\frac{n}2\Sigma^2* t^{\otimes 2}}t^{\otimes m}\, dt\nonumber\\
&=&\frac 1{(2\pi)^2\, n^{\frac m2+1}}\int_{[-b\sqrt{n},b\sqrt{n}]^2}e^{-it*\frac\ell{\sqrt{n}}}e^{-\frac{1}2\Sigma^2* t^{\otimes 2}}t^{\otimes m}\, dt\nonumber\\
&=&\frac {i^m}{n^{\frac m2+1}}
\Phi^{(m)}\left(\frac \ell {\sqrt{n}}\right)+o(n^{-K-\frac 14}).\label{INT}
\end{eqnarray}
Hence we have proved that
\begin{eqnarray*}
&\ &\left|
\mathbb E_{\bar\mu}\left[u\mathbf 1_{\{S_n=\ell\}}.v\circ \bar T^{n}\right]
-\sum_{m=0}^{2K-2}\frac{i^m}{m!}\frac{\Phi^{(m)}(\frac\ell{\sqrt{n}})}{n^{1+\frac m2}}*A_m(u,v)\right. \\
&\ &\left.-\sum_{m=0}^{2K-2} \sum_{j=4}^{4K-4-2m}\frac{i ^{m+j}}{m!\, j!\, n^{1+\frac{m+j}2}}\Phi^{(m+j)}\left(\frac \ell{\sqrt{n}}\right)*\left(
 A_m(u,v)\otimes (\lambda^n/a^n)_0^{(j)}\right)\right|\\
&\ &
\le \frac{ck^{2K-1}\Vert v\Vert_p\, \Vert u\Vert_\infty}{n^{K+\frac14}}\, ,
\end{eqnarray*}
and so \eqref{formuleTLL} using \eqref{INT} and the fact that
the uneven derivatives of $(\lambda/a)$ at 0 are null.
\end{proof}
\subsection{Generalization}

\begin{proposition}\label{TLL2}
Let $K$ be a positive integer.
Let $\xi_0\in\left(\max(\xi,\vartheta),1\right)$. There exists
$\mathfrak c_0>0$ such that, 
for every $u,v:\bar M\rightarrow \mathbb C$ 
dynamically Lipschitz continuous functions, with respect to $d_\xi$ with $\xi\in(0,1)$ and for every $\ell\in\ZZ^2$
\begin{eqnarray}
&\ &\left|
\mathbb E_{\bar\mu}\left[u\mathbf 1_{\{S_n=\ell\}}.v\circ \bar T^{n}\right]
-\sum_{m=0}^{2K-2}\frac 1{m!}\sum_{j=0}^{2K-2-m}
\frac {i^{m+2j}}{(2j)!}\frac{\Phi^{(m+2j)}\left(\frac \ell{\sqrt{n}}\right)}{n^{j+1+\frac{m}2}}*(A_{m}(u,v)\otimes(\lambda^n/a^n)_0^{(2j)})
\right|\nonumber\\
&\ &
\ \le\mathfrak c_0\frac{(\log n)^{4K-2}}{n^{K+\frac 14}}{\Vert v\Vert_{(\xi)}\, \Vert u\Vert_{(\xi)}}\, ,\label{eq:TLLGENE}
\end{eqnarray}
with $A_m(u,v)$ such that
\begin{equation}\label{Amlim}
\left|A_{m}(u,v)-\left(\mathbb E_{\bar\mu}[u.e^{it*S_n}.v\circ \bar T^n]/\lambda_t^n\right)^{(m)}_{|t=0}\right|
\le \mathfrak c_0\Vert u\Vert_{(\xi)}\, \Vert u\Vert_{(\xi)}\xi_0^{(\log n)^2}
\end{equation}
and $|A_m(u,v)|\le \mathfrak c_0\Vert u\Vert_{(\xi)}\Vert v\Vert_{(\xi)}$.
\end{proposition}
\begin{proof}
For every positive integer $k$, we define 
$$u_{k}:= \mathbb E_{\bar\mu}[u|\mathcal Z_{-k}^k]\quad\mbox{and}\quad
 v_{k}:= \mathbb E[v|\mathcal Z_{-k}^k]\, .$$
Note that
$$\Vert u-u_{k}\Vert_\infty\le   L_\xi (u)\xi^{k},\quad \Vert v-v_{k}\Vert_\infty\le   L_\xi (v)\xi^{k}
     \, ,$$
and
\[
\left|\mathbb E_{\bar\mu}\left[u\mathbf 1_{\{S_n=\ell\}}.v\circ \bar T^{n}\right]-\mathbb E_{\bar\mu}\left[u_{k}\mathbf 1_{\{S_n=\ell\}}.v_{k}\circ \bar T^{n}\right]\right|\le \Vert u\Vert_{(\xi)} \, \Vert v\Vert_{(\xi)}\xi^k\, .
\]
Now we take $k=k_n=\lceil (\log n)^2\rceil$.
Note that, for $n$ large enough, $n>3k_n$.
We set 
$$ A_{m,n}(u,v):=\left(\mathbb E_{\bar\mu}[u.e^{it*S_n}.v\circ \bar T^n]/\lambda_t^n\right)^{(m)}_{|t=0}\, .$$
Note that, for every integers $k,n>0$,
\begin{eqnarray*}
\left|A_{m,n}(u,v)-A_{m,n}(u_{k},v_{k})\right\Vert
&\le& \left\Vert \frac{\partial^m}{\partial t^m}\left(\frac {e^{it*S_n}}{\lambda_t^n}\right)_{|t=0}\right\Vert_{L^1(\bar\mu)}\Vert u\Vert_{(\xi)}\Vert v\Vert_{(\xi)}\xi^{k}\\
&\le& \tilde c_mn^m\Vert u\Vert_{(\xi)}\Vert v\Vert_{(\xi)}\xi^{k}\, .
\end{eqnarray*}
For every integers $n,n'$ such that $0<n\le n'\le 2n$, we have
\begin{eqnarray*}
&\ &\left|A_{m,n}(u,v)-A_{m,n'}(u,v)\right|\\
&\le&\left|A_{m,n}(u_{k_n},v_{k_n})-A_{m,n'}(u_{k_{n}},v_{k_{n}})\right|+(1+2^{m})\tilde c_mn^m\Vert u\Vert_{(\xi)}\Vert v\Vert_{(\xi)}\xi^{k_n} \\
&\le&K_m\Vert u\Vert_{(\xi)}\Vert v\Vert_{(\xi)}\xi_0^{k_n} \, ,
\end{eqnarray*}
due to \eqref{Am}.
Hence, we conclude that $(A_{m,n}(u,v))_n$ is a Cauchy sequence so that $A_m(u,v)$ is well defined and that
$$|A_{m}(u,v)-A_{m,n}(u,v)|\le K_m\Vert u\Vert_{(\xi)}\Vert v\Vert_{(\xi)}\sum_{j\ge 0}\xi_0^{k_{2^jn}}=O\left(\Vert u\Vert_{(\xi)}\, \Vert u\Vert_{(\xi)}\xi_0^{k_n}\right).$$
Since 
Applying Proposition \ref{TLL} to the couple $(u_{(k_{n})},v_{(k_{n})})$ leads to \eqref{eq:TLLGENE}.
\end{proof}
\subsection{Proofs of our main results}
\begin{theorem}\label{MAIN}
Let $f,g:M\rightarrow\mathbb R$ be two bounded observables
such that
\[
\sum_{\ell\in\mathbb Z^2}\left(\Vert f
      \mathbf 1_{\mathcal C_\ell}\Vert_{(\xi)}+\Vert g
      \mathbf 1_{\mathcal C_\ell}\Vert_{(\xi)}\right)<\infty\, .
\]
Then
\begin{eqnarray}
&\ &\int_Mf.g\circ T^n\, d\nu\nonumber\\
&=&\sum_{m=0}^{2K-2}\frac 1{m!}\sum_{j=0}^{2K-2-m}
\frac {i^{m+2j}}{(2j)!}\frac{\sum_{\ell,\ell'\in\mathbb Z^2}\Phi^{(m+2j)}\left(\frac {\ell'-\ell}{\sqrt{n}}\right)*(A_{m}(u_\ell,v_{\ell'}))}{n^{j+1+\frac{m}2}}*(\lambda^n/a^n)_0^{(2j)})
+o(n^{-K})\, ,\label{decorrelation1}
\end{eqnarray}
with $u_{\ell}(q,\vec v)=f(q+\ell,\vec v)$ and $v_{\ell}(q,\vec v)=f(q+\ell,\vec v)$ and with $A_m(u,v)$ given by \eqref{Amlim}.

If moreover, $\sum_{\ell\in\mathbb Z^2}|\ell|^{2K-2}(\Vert f\mathbf 1_{\mathcal C_\ell}\Vert_{(\xi)}+\Vert g\mathbf 1_{\mathcal C_\ell}\Vert_{(\xi)})<\infty$, then
\begin{eqnarray}
\int_Mf.g\circ T^n\, d\nu
=\sum_{L=0}^{K-1}\frac{\tilde c_L}{n^{1+L}}
 \sum_{j=0}^{2K-2-2L}(-1)^j
\frac {\Phi^{(2j+2L)}(0)}{(2j)!n^j}*(\lambda^n/a^n)_0^{(2j)}+o(n^{-K})\, \label{decorrelation2}
\end{eqnarray}
with
\begin{eqnarray*}
\tilde c_L(f,g)&:=&\sum_{r,m\ge 0\ :\ r+m=2L}\frac {i^m}{m!\, r!}\sum_{\ell,\ell'\in\mathbb Z^2}(\ell'-\ell)^{\otimes r}\otimes A_{m}(u_\ell,v_{\ell'})\, .
\end{eqnarray*}
\end{theorem}
Since $(\lambda^n/a^n)_0^{(2j)}=O(n^{j/2})$, we conclude that:
\begin{remark}\label{RQE}
Assume $\sum_{\ell\in\mathbb Z^2}|\ell|^{2K-2}(\Vert f\mathbf 1_{\mathcal C_\ell}\Vert_{(\xi)}+\Vert g\mathbf 1_{\mathcal C_\ell}\Vert_{(\xi)})<\infty$ and $\int_Mf.g\, d\nu=O(n^{-K})$.
Then 
$$\int_Mf.g\, d\nu=\frac{\Phi^{(2K-2)}(0)*\tilde c_{K-1}(f,g)}{n^{K}}+o(n^{-K}) \, ,$$
and $\tilde c_{K-1}(f,g)=\lim_{n\rightarrow +\infty}\frac {(-1)^{K-1}}{(2K-2)!}\sum_{\ell,\ell'\in\mathbb Z^2}\mathbb E_{\bar\mu}\left[u_\ell.\frac{\partial^{2K-2}}{\partial t^{2K-2}}\left(\lambda_t^{-n}e^{it*(S_n-(\ell'-\ell))}\right)_{|t=0}.
v_{\ell'}\circ\bar T^n\right]$.
\end{remark}
\begin{corollary}\label{coroMAIN}
Under the assumptions of Theorem \ref{MAIN} ensuring \eqref{decorrelation2}, using the fact that
$(\lambda/a)_0^{(2j)}=O(n^{j/2})$, as in Remarks \ref{DEVASYMP} and \ref{FORMULETLL2}, 
if $\sum_{\ell\in\mathbb Z^2}|\ell|^{4K-4}(\Vert f\mathbf 1_{\mathcal C_\ell}\Vert_{(\xi)}+\Vert g\mathbf 1_{\mathcal C_\ell}\Vert_{(\xi)})<\infty$,
the right hand side of \eqref{decorrelation2} can be rewritten
\[
n^{-\frac d2}
\sum_{\ell,\ell'\in\mathbb Z^2}\sum_{L=0}^{4K-4}\frac{\Phi^{(L)}(0)}{L!}
i^L\frac{\partial^L}{\partial t^L}\left(\mathbb E_{\bar\mu}\left[u_\ell. e^{it*\frac{(S_n-(\ell'-\ell))}{\sqrt{n}}}.v_{\ell'}\circ\bar T^n\right]e^{\frac{1}2\Sigma^2*t^{\otimes 2}}\right)_{|t=0}
+o(n^{-K})\, .
\]
\end{corollary}
\begin{proof}[Proof of Theorem \ref{MAIN}]
We have
$$\int_M f.g\circ T^n \, d\nu=\sum_{\ell,\ell'\in\mathbb Z^2}
   \mathbb E_{\bar \mu}[u_\ell \mathbf 1_{\{S_n=\ell'-\ell\}}v_{\ell'}\circ \bar T^n].$$
Hence, \eqref{decorrelation1} follows directly from Proposition \ref{TLL2}.
Due to the dominated convergence theorem,
\begin{eqnarray*}
&\ &\lim_{n\rightarrow+\infty}n^{K-1-\frac{m+j}2}\sum_{\ell,\ell'\in\mathbb Z^2}\left(\Phi^{(m+j)}\left(\frac{\ell'-\ell}{\sqrt{n}}\right)-\sum_{r=0}^{2K-2-m-(j/2)}\frac {\Phi^{(m+j+r)}(0)}{r!}*\left(\frac{\ell'-\ell}{\sqrt{n}}\right)^{\otimes r}\right)\\
&\ &\ \ \ \ \ \ \ *\left((\lambda^n/a^n)_0^{(j)}\otimes A_{m}(u_\ell,v_{\ell'})\right)=0\, ,
\end{eqnarray*}
(where we used \eqref{MAJO}) and to the fact that the uneven derivatives of $\Phi$ are null and that $\Phi^{(2k)}(0)=(-\Sigma^2)^{\otimes k}\Phi(0)$. Therefore
\begin{eqnarray*}
\int_Mf.g\circ T^n\, d\nu
&=&\sum_{m=0}^{2K-2}\sum_{r=0:r+m\in 2\mathbb Z}^{2K-2-m}\frac {\Phi(0)}{m!\, r!}\sum_{j=0}^{2K-2-m-r}
\frac {(-1)^{j}}{(2j)!}\left(\frac{(-\Sigma^{-2})^{\otimes (j+\frac{m+r}2)}}{n^{j+1+\frac{m+r}2}}*(\lambda^n/a^n)_0^{(2j)}\right)\nonumber\\
&\ &*\sum_{\ell,\ell'\in\mathbb Z^2}i^m(\ell'-\ell)^{\otimes r}\otimes A_{m}(u_\ell,v_{\ell'})
+o(n^{-K})\, ,
\end{eqnarray*}
which ends the proof of \eqref{decorrelation2}.
\end{proof}
\begin{proof}[Proof of Theorem \ref{PRINCIPAL}]
This comes from \eqref{decorrelation2} combined with the fact that
$(\lambda^n/a^n)_0^{(2j)}$ is a polynomial in $n$ of degree
bounded by $j/2$. 
\end{proof}
\begin{proof}[Proof of Theorem \ref{MAINbis}]
Due to \eqref{decorrelation2} of Theorem \ref{MAIN}, we obtain 
\eqref{MAIN2} with
$$ \tilde{\mathfrak A}_2(f,g)=\mathfrak  a_{2,0,0}(f,g)+\mathfrak  a_{0,2,0}(f,g)+\mathfrak  a_{1,1,0}(f,g)\, ,
$$
where $a_{m,r,j}(f,g)$ corresponds to the contribution of the $(m,r,j)$-term in the sum of the right hand side of \eqref{decorrelation2}. Moreover, due to Proposition \ref{Amn!!},
\begin{eqnarray*}
\mathfrak  a_{2,0,0}(f,g)&= &\sum_{\ell,\ell'\in\mathbb Z^2}A_2(u_\ell,v_{\ell'})\\
&=&-\lim_{n\rightarrow+\infty}\left\{   \int_Mf\, d\nu\, \sum_{j,m= -n}^{-1}
\int_M g.(\kappa\circ T^j\otimes \kappa\circ T^m-\mathbb E_{\bar\mu}[\kappa\circ \bar T^j\otimes \kappa\circ\bar T^m])]\, d\nu\right.\\
&\ &+\int_M g\, d\nu\, 
\sum_{j,m= 0}^{n-1}\int_M f .[\kappa\circ T^j\otimes\kappa\circ T^m
    -\mathbb E_{\bar\mu}[\kappa\circ T^j\otimes\kappa\circ T^m]]\, d\nu\nonumber\\
&\ & +2\sum_{r= 0}^{n-1}\int_M f.\kappa \circ T^r\, d\nu
     \otimes\sum_{m= -n}^{-1}\int_M g.\kappa\circ T^m\, d\nu\nonumber\\
&\ &   \left. +\int_Mf\, d\nu\, \int_Mg\, d\nu (\mathbb E_{\bar\mu}[S_n^{\otimes 2}]-n\Sigma^2)\right\}\, ,
\end{eqnarray*}
\[
\mathfrak  a_{0,2,0}(f,g) =-\sum_{\ell,\ell'\in\mathbb Z^2}A_0(u_\ell,v_{\ell'}).(\ell'-\ell)^{\otimes 2}=-
\sum_{\ell,\ell'\in\mathbb Z^2}(\ell'-\ell)^{\otimes 2}\int_{\mathcal C_\ell}f\, d\nu\, \int_{\mathcal C_{\ell'}}g\, d\nu\, ,
\]
\begin{eqnarray*}
\mathfrak  a_{1,1,0}(f,g)&=&-2i\sum_{\ell,\ell'\in\mathbb Z^2}A_1(u_\ell,v_{\ell'})\otimes(\ell'-\ell)\\
&= &2\lim_{n\rightarrow +\infty}\left\{\sum_{\ell,\ell'\in\mathbb Z^2}\int_{\mathcal C_{\ell'}}g\, d\nu\sum_{r= 0}^{n-1}\int_{\mathcal C_\ell}f.((\ell'-\ell)\otimes \kappa\circ T^r)\, d\nu\right.\\
&\ &\left.+\sum_{\ell,\ell'\in\mathbb Z^2}\int_{\mathcal C_{\ell}}f\, d\nu\sum_{m=-n}^{-1}\int_{\mathcal C_{\ell'}}g.((\ell'-\ell)\otimes \kappa\circ T^m)\, d\nu\right\}\, .
\end{eqnarray*}
For the contribution of the term with $(m,r,j)=(0,0,2)$, note that
$$(\lambda^n/a^n)^{(4)}_{0}=n(\lambda/a)_0^{(4)}=n(\lambda_0^{(4)}-3(\Sigma^2)^{\otimes 2}).$$
Moreover, due to Proposition \ref{lambda04},
\[
\lambda_0^{(4)}-3(\Sigma^2)^{\otimes 2}=
\lim_{n\rightarrow +\infty}\frac{\mathbb E_{\bar\mu}[S_n^{\otimes 4}]-3n^2(\Sigma^2)^{\otimes 2}}{n}+6\Sigma^2\otimes B_0=\Lambda_4\, .
\]
Note that
\begin{eqnarray*}
\mathfrak  a_{2,0,0}(f,g)&= &-\lim_{n\rightarrow+\infty}\left\{   \int_Mf\, d\nu\, \int_M g((\mathcal I_0-\mathcal I_{-n})^{\otimes 2}-\mathbb E_{\bar\mu}[S_n^{\otimes 2}])\, d\nu\right.\\
&\ &+\int_M g\, d\nu\, 
\int_M f((\mathcal I_n-\mathcal I_{0})^{\otimes 2}-\mathbb E_{\bar\mu}[S_n^{\otimes 2}])\, d\nu\nonumber\\
&\ & +2\int_Mf(\mathcal I_n-\mathcal I_0)\, d\nu
     \otimes\int_Mg(\mathcal I_0-\mathcal I_{-n})\, d\nu\nonumber\\
&\ &   \left. -\int_Mf\, d\nu\, \int_Mg\, d\nu \, \mathfrak B_0\right\}\, ,
\end{eqnarray*}
\begin{eqnarray*}
\mathfrak  a_{0,2,0}(f,g) &=&
-\int_M f.\mathcal I_0^{\otimes 2}\, d\nu\, \int_Mg\, d\nu-
\int_Mf\, d\nu\, \int_M g.\mathcal I_0^{\otimes 2}\, d\nu+2
\int_Mf\mathcal I_0\, d\nu\otimes\int_Mg\mathcal I_0\, d\nu
\end{eqnarray*}
and
\begin{eqnarray*}
\mathfrak  a_{1,1,0}(f,g)
&= &\lim_{n\rightarrow +\infty}\left\{2\int_M g\mathcal I_0\, d\nu\otimes 
\int_M f(\mathcal I_n-\mathcal I_0)\, d\nu-2\int_M g\, d\nu\, 
\int_M f.\mathcal I_0\otimes(\mathcal I_n-\mathcal I_0)\, d\nu\right.\\
&\ &\left.+2\int_Mf\, d\nu\, \int_Mg.\mathcal I_0\otimes(\mathcal I_0-\mathcal I_{-n})\, d\nu-2\int_Mf\mathcal I_0\, d\nu\otimes
    \int_Mg(\mathcal I_0-\mathcal I_{-n})\, d\nu\right\}\, .
\end{eqnarray*}
Hence we have proved \eqref{MAIN2} with
\[
\tilde {\mathfrak A_2}(f,g):= 
-\int_Mf\, d\nu \, \tilde{\mathfrak B}_2^-(g)-\int_Mg\, d\nu\, \tilde{\mathfrak B}_2^+(f)+\int_Mf\, d\nu\int_Mg\, d\nu\, {\mathfrak B}_0+2\, {\mathfrak B}_1^+(f)\otimes{\mathfrak B}_1^-(g)\, ,
\]
with
\begin{eqnarray*}
\tilde{\mathfrak B}_2^+(f)&:=&\lim_{m\rightarrow +\infty} \int_Mf\left(\mathcal I_m^{\otimes 2}-\mathbb E[S_m^{\otimes 2}]\right)\, d\nu  \, ,
\end{eqnarray*}
\begin{eqnarray*}
\mathfrak B_2^-(g)&:=&\lim_{m\rightarrow -\infty}\int_M g\left(\mathcal I_m^{\otimes 2}-\mathbb E[S_m^{\otimes 2}]\right)\, d\nu\, .
\end{eqnarray*}
\end{proof}

\begin{remark}\label{MAINter}
Let $f,g:M\rightarrow\mathbb R$ be two bounded observables
such that
\begin{equation}
\sum_{\ell\in\mathbb Z^2}|\ell|^4\left(\Vert f
      \mathbf 1_{\mathcal C_\ell}\Vert_{(\xi)}+\Vert g
      \mathbf 1_{\mathcal C_\ell}\Vert_{(\xi)}\right)<\infty\,  
\end{equation}
Assume moreover that $\int_Mf\, d\nu\, \int_Mg\, d\nu=0$ and that
$\tilde {\mathfrak A}_2(f,g)=0$.
Due to Remark \ref{RQE},
\begin{eqnarray*}
&\ &\int_Mf.g\circ T^n\, d\nu\\
&\ &= \frac {(\Sigma^{-2})^{\otimes 2}} {2\pi\sqrt{\det \Sigma^2}n^3}
       *\sum_{\ell,\ell'\in\mathbb Z^2}\left(\frac{A_4(u_\ell,v_{\ell'})}{24}
+\frac {A_0(u_\ell,v_{\ell'})}{24}({\ell'-\ell})^{\otimes 4}+\frac {i\, A_1(u_\ell,v_{\ell'})}6\otimes(\ell'-\ell)^{\otimes 3}\right. \\
&\ &\left.-\frac 14A_2(u_\ell,v_{\ell'})\otimes(\ell'-\ell)^{\otimes 2}
- \frac i6 A_3(u_\ell,v_{\ell'})\otimes(\ell'-\ell)\right)
+o(n^{-3})\, ,
\end{eqnarray*}
where $u_\ell(q,\vec v):=f(q+\ell,\vec v)$
and $v_\ell(q,\vec v):=g(q+\ell,\vec v)$.
\end{remark}
\begin{proof}[Proof of Proposition \ref{casparticulier}]
We apply Remark \ref{MAINter}. Using
the definitions of $A_0$ and $A_1$, we observe that
$$\forall \ell,\ell'\in\mathbb Z^2,\quad A_0(u_\ell,v_{\ell'})=A_1(u_{\ell},v_{\ell'})=0 $$
(since $\mathbb E_{\bar\mu}[u_\ell]=\mathbb E_{\bar\mu}[v_{\ell'}]=0$)
and
$$\sum_{\ell,\ell'\in\mathbb Z^2}A_4(u_{\ell},v_{\ell'})=A_4\left(\sum_{\ell\in\mathbb Z^2}u_\ell,\sum_{\ell'\in\mathbb Z^2}v_{\ell'}\right) =0\, .$$
Moreover
$$\sum_{\ell,\ell'\in\mathbb Z^2}A_3(u_\ell,v_{\ell'})\otimes(\ell'-\ell)
= \sum_{\ell,\ell'\in\mathbb Z^2}h_\ell q_{\ell'}A_3(f_0,g_0)\otimes(\ell'-\ell)=0$$
since $\sum_{\ell\in\mathbb Z^2} h_\ell=\sum_{\ell}q_\ell=0$.
Therefore
\begin{eqnarray*}
&\ &\int_Mf.g\circ T^n\, d\nu\\
&\ &= -\frac 14\frac {(\Sigma^{-2})^{\otimes 2}} {2\pi\sqrt{\det \Sigma^2}n^3}
       *\sum_{\ell,\ell'\in\mathbb Z^2}A_2(u_\ell,v_{\ell'})\otimes(\ell'-\ell)^{\otimes 2}
+o(n^{-3})\\
&\ &= \frac 12\frac {(\Sigma^{-2})^{\otimes 2}} {2\pi\sqrt{\det \Sigma^2}n^3}
       *\sum_{\ell,\ell'\in\mathbb Z^2}h_\ell q_{\ell'}A_2(f_0,g_0)\otimes \ell\otimes\ell'
+o(n^{-3})\\
&\ &= \frac 12\frac {(\Sigma^{-2})^{\otimes 2}} {2\pi\sqrt{\det \Sigma^2}n^3}
     A_2(f_0,g_0)\otimes \sum_{\ell\in\mathbb Z^2}h_\ell.\ell \otimes\sum_{\ell'\in\mathbb Z^2} q_{\ell'}.\ell' 
+o(n^{-3})\\
&\ &= -\frac {(\Sigma^{-2})^{\otimes 2}} {2\pi\sqrt{\det \Sigma^2}n^3}*\left(
     \sum_{j\ge 0}\mathbb E_{\bar\mu}[f_0.\kappa\circ\bar T^j]\otimes
    \sum_{m\le -1}\mathbb E_{\bar\mu}[g_0.\kappa\circ \bar T^m]\otimes
 \sum_{\ell\in\mathbb Z^2}h_\ell.\ell \otimes\sum_{\ell'\in\mathbb Z^2} q_{\ell'}.\ell' \right)
+o(n^{-3})\\
&\ &= -\frac {(\Sigma^{-2})^{\otimes 2}} {2\pi\sqrt{\det \Sigma^2}n^3}*\left(
     \sum_{j\ge 0}\int_Mf\mathcal I_0\otimes \kappa\circ T^j\, d\nu\otimes \sum_{m\le -1}\int_Mg\mathcal I_0\otimes \kappa\circ T^m\, d\nu\right)+o(n^{-3})\, .
\end{eqnarray*}
\end{proof}
\section{Proof of the mixing result in the infinite horizon case}\label{sec:infinite}
\begin{proof}[Proof of Theorem \ref{horizoninfini}]
In \cite{SV2}, Sz\'asz and Varj\'u implemented the Nagaev-Guivarc'h perturbation method via the Keller-Liverani theorem \cite{KL} to prove that Hypothesis \ref{HHH} holds true for the dynamical system $(\hat M,\hat\mu,\hat T)$ with the Young Banach space $\mathcal B$,
with $\mathcal B_0:=\mathbb L^1(\hat \mu)$ and with $\lambda$
having the following expansion:
$$ \lambda_t-1\sim \Sigma_\infty^2*(t^{\otimes 2})\log |t|\, .$$
Hence Hypothesis \ref{HHH1} holds also true, with $\Theta_n=\sqrt{n\log n}\, Id$ and with $Y$ a gaussian random variable
with distribution $\mathcal N(0,\Sigma_\infty^2)$ with density
function $\Phi(x)=\exp(-\frac 12(\Sigma_\infty^2)^{-1}*x^{\otimes 2})/
    (2\pi\sqrt{\det\Sigma_\infty^2})$.
Let $k_n:=\lceil \log^2n\rceil$.
Let $u_{n}(x)$ and $v_{n}(x)$ correspond to the conditional 
expectation of respectively $f$ and $g$ over the connected component of $M\setminus\bigcup_{m=-k_n}^{k_n}T^{-m}\mathcal S_0$ containing $x$.
First note that
\begin{equation}\label{horinf1}
\int_M f.g\circ T^n\, d\nu =\int_M u_n.v_n\circ T^n\, d\nu+O\left(
  \left(L_\xi(f)\int_M|g|\, d\nu+L_\xi(g)\int_M|f|\, d\nu\right)\xi^{k_n}\right)\, .
\end{equation}
As noticed in Proposition \ref{pro:pertu2b},  
there exist $\hat f_n,\hat g_n:\hat M\times\mathbb Z^2\rightarrow\mathbb C$
such that
$$\forall \tilde x\in\tilde M,\quad\hat f_n (\hat \pi(\tilde x),\ell)=u_{n}(\bar T^{k_n}(\tilde\pi (\tilde x))+\ell)\, ,$$
$$\forall \tilde x\in\tilde M,\quad\hat g_n (\hat \pi(\tilde x),\ell)=v_{n}(\bar T^{k_n}(\tilde\pi (\tilde x))+\ell)\, ,$$
with the notation $(q,\vec v)+\ell=(q+\ell,\vec v)$ for every
$(q,\vec v)\in \bar M$.
For $n$ large enough, $n>3k_n$ and, due to \eqref{EqClef},
\begin{eqnarray*}
&\ &\int_M\, u_{n}.v_{n}\circ T^n\, d\nu
=\sum_{\ell,\ell'\in\mathbb Z^2}\mathbb E_{\bar\mu}[u_{n}(\cdot+\ell).\mathbf 1_{S_n=\ell'-\ell}.v_{n}(\bar T^n(\cdot)+\ell')]\nonumber\\
&\ &=\sum_{\ell,\ell'\in\mathbb Z^2}\frac 1{(2\pi)^2}\int_{[-\pi,\pi]^2}e^{-it*(\ell'-\ell)}\mathbb E_{\bar\mu}[u_{n}(\cdot+\ell).e^{it*S_n}
   .v_{n}(\bar T^n(\cdot)+\ell')]\nonumber\\
&\ &=\sum_{\ell,\ell'\in\mathbb Z^2}\frac 1{(2\pi)^2}\int_{[-\pi,\pi]^2}e^{-it*(\ell'-\ell)}\mathbb E_{\hat\mu}[\hat G_{n,t}(\cdot,\ell')\hat P_t^{n-2k_n}\hat P^{2k_n}(\hat F_{n,t}(\cdot,\ell))]\, dt\, ,\nonumber\end{eqnarray*}
where  $\hat F_{n,t},\hat G_{n,t}:\hat M\rightarrow\mathbb Z^2\rightarrow\mathbb C$ are the functions defined by 
$$\hat F_{n,t}(\hat x,\ell):=\hat f_n(\hat x,\ell).e^{it*\hat S_{k_n}(\hat T^{k_n}(\hat x))},$$ 
$$\hat G_{n,t}(\hat x,\ell):=\hat g_n(\hat x,\ell).e^{it*\hat S_{k_n}(\hat x)}.$$
Moreover $\sup_{n,t}\Vert \hat P^{2k_n}\hat F_{n,t}(\cdot,\ell)\Vert\le (1+2\beta^{-1})\Vert f\mathbf 1_{\mathcal C_\ell}\Vert_\infty$.
Hence, due to Hypothesis \ref{HHH},
\begin{eqnarray*}
&\ &\int_M\, u_{n}.v_{n}\circ T^n\, d\nu\\
&\ &=O(\vartheta^{n-2k_n})+\sum_{\ell,\ell'\in\mathbb Z^2}\frac 1{(2\pi)^2}\int_{[-\pi,\pi]^2}e^{-it*(\ell'-\ell)}\mathbb E_{\hat\mu}[\hat G_{n,t}(\cdot,\ell')\lambda_t^{n-2k_n}\Pi_t \hat P^{2k_n}(\hat F_{n,t}(\cdot,\ell))]\, dt\nonumber\\
&\ &=O(\vartheta^{n-2k_n})+\sum_{\ell,\ell'\in\mathbb Z^2}\frac 1{\mathfrak a_n^2(2\pi)^2}\int_{[-\mathfrak a_n\pi,\mathfrak a_n\pi]^2}e^{-iu*\frac{\ell'-\ell}{\mathfrak a_n}}\mathbb E_{\hat\mu}[\hat G_{n,u/\mathfrak a_n}(\cdot,\ell')\lambda_{u/\mathfrak a_n}^{n-2k_n}\Pi_{u/\mathfrak a_n} \hat P^{2k_n}(\hat F_{n,u/\mathfrak a_n}(\cdot,\ell))]\, du\nonumber\\
&\ &=o(\mathfrak a_n^{-2})+\sum_{\ell,\ell'\in\mathbb Z^2}\frac 1{\mathfrak a_n^2(2\pi)^2}\int_{[-\mathfrak a_n\pi,\mathfrak a_n\pi]^2}\!\!\!\!\!\!\!
  \mathbb E_{\hat\mu}[\hat G_{n,0}(\cdot,\ell')e^{-\frac12 \Sigma_\infty^2*u^{\otimes 2}}\Pi_0 \hat P^{2k_n}(\hat F_{n,u/\mathfrak a_n}(\cdot,\ell))]\, du\, \\
  &\ &=o(\mathfrak a_n^{-2})+\sum_{\ell,\ell'\in\mathbb Z^2}\frac 1{\mathfrak a_n^2(2\pi)^2}\int_{[-\mathfrak a_n\pi,\mathfrak a_n\pi]^2}\!\!\!\!\!\!\!
  \mathbb E_{\hat\mu}[\hat G_{n,0}(\cdot,\ell')]e^{-\frac12 \Sigma_\infty^2*u^{\otimes 2}} \mathbb E_{\hat\mu}[\hat F_{n,0}(\cdot,\ell))]\, du\, ,
\end{eqnarray*}
where we used the change of variable $u=\mathfrak a_n\, t$ 
with $\mathfrak a_n:=\sqrt{(n-2k_n)\log(n-2k_n)}$, and twice
the dominated convergence theorem. Therefore
\[
\int_M\, u_{n}.v_{n}\circ T^n\, d\nu=\frac 
{\Phi\left(0\right)}{\mathfrak a_n^2(2\pi)^2}\int_Mu_n\, d\nu\,
\int_M v_n\, d\nu+o(\mathfrak a_n^{-2})\, .
\]
The conclusion of the theorem follows from this last formula combined with \eqref{horinf1} and with the facts that
$\mathfrak a_n^2\sim n\log n$
and that
$$\int_Mu_n\, d\nu\, \int_Mv_n\, d\nu=\int_Mf\, d\nu \, \int_Mg\, d\nu\, ,$$
due to the dominated convergence theorem.
\end{proof}
\begin{appendix}
\section{Billiard with finite horizon: about the coefficients $A_{m}$}\label{sec:coeff}
Let $\mathcal W^s$ (resp. $\mathcal W^u$) be the set of stable (resp. unstable) $H$-manifolds.
In \cite{Chernov}, Chernov defines two separation times $s_+$ and $s_-$ which are dominated by $s$ and 
such that, for every positive integer $k$,
$$\forall W^u\in\mathcal W^u,\ \forall \bar x,\bar y\in W^u,\quad s^+(\bar T^{-k}\bar x,\bar T^{-k}\bar y)=s^+(x,y)+k,$$
$$\forall W^s\in\mathcal W^s,\ \forall \bar x,\bar y\in W^s,\quad s^-(\bar T^k\bar x,\bar T^{k}\bar y)=s^-(x,y)+k.$$

\begin{proposition}[\cite{Chernov}, Theorem 4.3 and remark after]\label{decoChernov}
There exist $C_0>0$ and $\vartheta_0\in(0,1)$ such that, for every positive
integer $n$, for every bounded measurable $u,v:\bar M\rightarrow\mathbb R$,
$$\left |\mathbb E_{\bar\mu}[u.v\circ\bar T^n]
    -\mathbb E_{\bar\mu}[u]\mathbb E_{\bar\mu}[v] \right\|
       \le C_0\left(L_u^+\Vert v\Vert_\infty+L_v^-\Vert u\Vert_\infty+\Vert u\Vert_\infty\Vert v\Vert_\infty\right)\vartheta_0^n\, ,
$$
with
$$ L_u^+:=\sup_{W^u\in \mathcal W^u}\sup_{x,y\in W^u,\, x\ne y}(|u(x)-u(y)|\xi^{-\mathbf{s}_+(x,y)}),$$
and
$$L_v^-:=\sup_{W^s\in \mathcal W^s}\sup_{x,y\in W^s,\, x\ne y}(|v(x)-v(y)|\xi^{-\mathbf{s}_-(x,y)})\, .$$
\end{proposition}
Note that
$$ L_u^+\le L_\xi(u\mathbf 1_{\bar M }),\quad L_u^-\le L_\xi(u\mathbf 1_{\bar M })\, ,$$
$$L^+_{u\circ \bar T^{-k}}\le L_u^+\xi^k\quad \mbox{and}\quad
    L^-_{v\circ \bar T^{k}}\le L_v^-\xi^k\, . $$

We will set $\tilde u:=u-\mathbb E_{\bar\mu}[u]$ and $\tilde v:=v-\mathbb E_{\bar\mu}[v]$.
We will express the terms $A_m(u,v)$ for $m\in\{1,2,3,4\}$ in terms of the follwing quantities:
$$B_1^+(u):= \sum_{j\ge 0}\mathbb E_{\bar\mu}[u.\kappa\circ T^j] \, ,\quad B_1^-(v):=\sum_{m\le -1}\mathbb E_{\bar\mu}[v.\kappa\circ \bar T^m]\, ,$$
$$B_2^+(u):=\sum_{j,m\ge 0}\mathbb E_{\bar\mu}[\tilde u.\kappa\circ \bar T^j\otimes\kappa\circ\bar T^m]\, ,\quad B_2^-(v):=\sum_{j,m\le -1}
   \mathbb E_{\bar\mu}[\tilde v.\kappa\circ \bar T^j\otimes \kappa\circ \bar T^m]\, ,$$
$$B_0^-(v):=\sum_{k\le -1}|k|
\mathbb E_{\bar\mu}[\tilde v.\kappa\circ\bar T^k]\, ,\quad B_0^+(u)=\sum_{k\ge 0}k
\mathbb E_{\bar\mu}[\tilde u.\kappa\circ\bar T^k]\, ,$$
$$B_0:=B_0^-(\kappa)+B_0^+(\kappa)=\sum_{m\in\mathbb Z}|m|\mathbb E_{\bar\mu}[\kappa\otimes\kappa\circ\bar T^m]\, , $$
$$B_{0,2}^+(u):=\sum_{k,m\ge 0}\max(k,m)
\mathbb E_{\bar\mu}[\tilde u.\kappa \circ \bar T^k\otimes \kappa\circ\bar T^m]\, .$$
$$B_{0,2}^-(v):=
\sum_{k,m\ge 1}\max(k,m)
\mathbb E_{\bar\mu}[\tilde v.\kappa \circ \bar T^{-k}\otimes \kappa\circ\bar T^{-m}]\, ,$$
\begin{eqnarray*}
B_3^+(u)&:=&\sum_{k,r, m\ge 0} \mathbb E_{\bar\mu}[
     \tilde u.\kappa\circ \bar T^{\min(k,r,m)}\\
&\ &\left(\kappa\circ \bar T^{\max(k,r,m)}\otimes\kappa\circ\bar T^{med(k,r,m)}-\mathbb E_{\bar\mu}[\kappa\circ \bar T^{\max(k,r,m)}\otimes\kappa\circ\bar T^{med(k,r,m)}]\right)]\, ,
\end{eqnarray*}
\begin{eqnarray*}
B_3^-(v)&:=&\sum_{m, r, s\le -1}\mathbb E_{\bar\mu}[
\tilde v .\kappa\circ\bar T^{\max(m,r,s)}\otimes\\
&\ & \left(\kappa\circ \bar T^{\min(m,r,s)} \otimes\kappa\circ\bar T^{med(m,r,s)}-\mathbb E_{\bar\mu}[\kappa\circ \bar T^{\min(m,r,s)} \otimes\kappa\circ\bar T^{med(m,r,s)}]\right)]\, ,
\end{eqnarray*}
with $med(m,r,s)$ the mediane of $(m,r,s)$.
\begin{proposition}\label{Amn!!}
Let $u,v:\bar M\rightarrow \mathbb C$ be two dynamically Lipschitz continuous functions, with respect to $d_\xi$ with $\xi\in(0,1)$.
Then
\begin{eqnarray}
 A_{0}(u,v)&=&\mathbb E_{\bar\mu}[u].\mathbb E_{\bar\mu}[v]\\
 A_{1}(u,v)&=& i\lim_{n\rightarrow +\infty}\mathbb E_{\bar\mu}[u.S_n.v\circ\bar T^n]=i\,B_1^+(u)\mathbb E_{\bar \mu}[v]+i\,B_1^-(v)\mathbb E_{\bar\mu}[u]\\
 A_{2}(u,v)&=&\lim_{n\rightarrow +\infty}(n\, \mathbb E_{\bar\mu}[u]\mathbb E_{\bar\mu}[v]\Sigma^2-\mathbb E_{\bar\mu}[u.S_n^{\otimes 2}.v\circ \bar T^n])\\
&=&-2\, B_1^+(u)\otimes B_1^-(v)-\mathbb E_{\bar\mu}[v]B_2^+(u)-\mathbb E_{\bar\mu}[u]B_2^-(v)  +\mathbb E_{\bar\mu}[u]\mathbb E_{\bar\mu}[v]\, B_0\, ,\label{EEEE}\label{Pi"0}
\end{eqnarray}

Moreover
\begin{eqnarray}
A_{3}(u,v)&=&\lim_{n\rightarrow +\infty}\left(3in\Sigma^2\otimes
\mathbb E_{\bar\mu}[u.S_n.v\circ\bar T^n]-i\mathbb E_{\bar\mu}[u.S_n^{\otimes 3}.v\circ\bar T^n]\right)\nonumber\\
&=&3 A_1(u,v)\otimes B_0+3i\Sigma^2
   \otimes\left(\mathbb E_{\bar\mu}[u]B_0^-(v)+ \mathbb E_{\bar \mu}[ v]B_0^+(u)\right)\nonumber\\
&\ &-i\mathbb E_{\bar\mu}[v]B_3^+(u)-i\mathbb E_{\bar\mu}[u]B_3^-(v)-3iB_2^-(v)\otimes B_1^+(u)-3iB_2^+(u)\otimes
B_1^-(v)\label{Pi'''0}
\end{eqnarray}
and
\begin{eqnarray*}
A_{4}(u,v)&=&
\lim_{n\rightarrow +\infty}\mathbb E_{\bar\mu}[u.S_n^{\otimes 4}.v\circ\bar T^n]+(\lambda^{-n})_0^{(4)}\mathbb E_{\bar\mu}[u]\mathbb E_{\bar\mu}[v]+6n\Sigma^2\otimes\mathbb E_{\bar\mu}[u.S_n^{\otimes 2}.v\circ\bar T^n]\\
&=&6B_0A_2(u,v)-6\Sigma^2\otimes\left(\mathbb E_{\bar \mu}[ u] B_{0,2}^-(v))-6\mathbb E_{\bar \mu}[ v] B_{0,2}^+(u)\right)\\
&\ &+\mathbb E_{\bar \mu}[ u]\mathbb E_{\bar \mu}[ v](A_{4}(\mathbf 1,\mathbf 1)-6B_0^{\otimes 2})\\
&\ &-12\Sigma^2\otimes 
(B_1^+(u)\otimes B_0^-(v)+B_1^-(v)\otimes B_0^+(u)
          -B_1^+(u)\otimes B_1^-(v))\\
&\ &+4B_1^+(u)\otimes B_3^-(v)+6 B_2^+(u)\otimes B_2^-(v) +4\  B_1^-(v)\otimes B_3^+(u)\, .\label{formuleA4}
\label{Pi40}
\end{eqnarray*}
\end{proposition}
\begin{proof}
As in the proof of Theorem \ref{TLL2}, we set 
$$ A_{m,n}(u,v):=\left(\mathbb E_{\bar\mu}[v.e^{it*S_n}.u\circ \bar T^n]/\lambda_t^n\right)^{(m)}_{|t=0}.$$
We will only use Proposition \ref{decoChernov} and the fact that $\lambda_t=1-\frac 12\Sigma^2*t^{\otimes 2}+\frac 1{4!}\lambda_0^{(4)}*t^{\otimes 4}+o(|t|^4)$ to compute $A_m(u,v)=\lim_{n\rightarrow +\infty}A_{m,n}(u,v)$. 
\begin{itemize}
\item
First we observe that
$
 A_{0,n}(u,v)=\mathbb E_{\bar\mu}[u.v\circ\bar T^n]
$
and we apply Proposition \ref{decoChernov}.
\item
Second,
\begin{eqnarray*}
 A_{1,n}(u,v)&=&i\, \mathbb E_{\bar\mu}[u.S_n.v\circ\bar T^n]
= i\, \sum_{k=0}^{n-1}\mathbb E_{\bar\mu}[u.\kappa\circ\bar T^k.v\circ\bar T^n]\\
&=& i\, \sum_{k=0}^{\lfloor n/2\rfloor}\mathbb E_{\bar\mu}[u.\kappa\circ \bar T^k]\mathbb E_{\bar\mu}[v]+i\, \sum_{\lfloor n/2\rfloor+1}^{n-1}\mathbb E_{\bar\mu}[u]\mathbb E_{\bar\mu}[v.\kappa\circ \bar T^{-(n-k)}]+O\left(n\vartheta_0^{n/2}\Vert u\Vert_{(\xi)}\Vert u\Vert_{(\xi)}\right)\\
&=& i\, \mathbb E_{\bar\mu}[v]\sum_{k\ge 0}
\mathbb E_{\bar\mu}[u.\kappa\circ \bar T^k]+i\, \mathbb E_{\bar\mu}[u]\sum_{m\le -1}\mathbb E_{\bar\mu}[v.\kappa\circ \bar T^{m}]+O\left(n\vartheta_0^{n/2}\Vert u\Vert_{(\xi)}\Vert u\Vert_{(\xi)}\right),
\end{eqnarray*}
where we used several times  Proposition \ref{decoChernov},
combined with the fact that $\mathbb E_{\bar\mu}[\kappa]=0$.
\item
Third,
\begin{eqnarray}
 A_{2,n}(u,v)&=&-\mathbb E_{\bar\mu}[ u . S_n^{\otimes 2}. v\circ \bar T^n] +n \Sigma^2\mathbb E_{\bar\mu}[ u]\mathbb E_{\bar\mu}[ v]\label{formuleA2n}\\
&=&-\sum_{k,m=0}^{n-1}\mathbb E_{\bar\mu}[ u.(\kappa\circ \bar T^k\otimes  \kappa\circ\bar T^m). v\circ \bar T^n] +n \Sigma^2\mathbb E_{\bar\mu}[ u]\mathbb E_{\bar\mu}[ v]\nonumber\\
&=&-\sum_{k,m=0}^{n-1}\mathbb E_{\bar\mu} [ \tilde u\kappa\circ\bar T^k\otimes\kappa\circ \bar T^m.\tilde v\circ\bar T^n]\nonumber\\
&\ &-\sum_{k,m=0}^{n-1}\left(\mathbb E_{\bar\mu}[ u]\mathbb E_{\bar\mu}[ \kappa\circ\bar T^k\otimes \kappa\circ\bar T^m
         \tilde v\circ\bar T^n]+\mathbb E_{\bar\mu}[\tilde u.\kappa\circ\bar T^k\otimes\kappa\circ \bar T^m]\mathbb E_{\bar\mu}[ v]\right)\nonumber\\
&\ &\ \ \ \ \ +(n\Sigma^2-\sum_{k,m=0}^{n-1}\mathbb E_{\bar\mu}[ \kappa\circ \bar T^k\otimes \kappa\circ \bar T^m])\mathbb E_{\bar\mu}[ u]\mathbb E_{\bar\mu}[ v]
\end{eqnarray}
\begin{itemize}
\item On the first hand
\begin{eqnarray*}
n\Sigma^2-\sum_{k,m=0}^{n-1}\mathbb E_{\bar\mu}[\kappa\circ \bar T^k\otimes\kappa\circ \bar T^m]
&=&n\sum_{k\in\mathbb Z}\mathbb E_{\bar\mu}[ \kappa\otimes\kappa\circ \bar T^k]
   -\sum_{k=-n}^n(n-|k|)\mathbb E_{\bar\mu}[ \kappa\otimes\kappa\circ \bar T^k]\\
&=&\sum_{k\in\mathbb Z}\min(n,|k|)\mathbb E_{\bar\mu}[\kappa\otimes\kappa\circ \bar T^k],
\end{eqnarray*}
which converges to $\sum_{k\in\mathbb Z}|k|\mathbb E_{\bar\mu}[\kappa\otimes\kappa\circ\bar T^k]$.
\item On the second hand,  for $0\le k\le m\le n$, due to Proposition \ref{decoChernov} (treating separately the cases $k\ge n/3$, $m-n\ge n/3$ et $ n-m\ge n/3$),
\begin{equation}\label{Esp4termes}
\mathbb E_{\bar\mu} [ \tilde u.\kappa\circ \bar T^k\otimes
\kappa
\circ\bar T^m.\tilde v\circ \bar T^n]=\mathbb E_{\bar\mu}[\tilde u.\kappa \circ \bar T^k]\otimes\mathbb E_{\bar\mu}[ \tilde v. \kappa\circ \bar T^{n-m}] +O(\Vert u\Vert_{(\xi)}\Vert v\Vert_{(\xi)}\vartheta_0^{n/3}).
\end{equation}
Analogously
\begin{equation}
\mathbb E_{\bar\mu}[ \kappa\circ \bar T^k\otimes\kappa\circ\bar T^m
         \tilde v\circ\bar T^n]=O(\Vert v\Vert_{(\xi)}\vartheta_0^{(n-k)/2})
\end{equation}
\begin{equation}\label{Esp3termes}
\mathbb E_{\bar\mu}[\tilde u.\kappa\circ \bar T^k\otimes\kappa\circ \bar T^m]=O(\Vert u\Vert_{(\xi)}\vartheta_0^{m/2})\, .
\end{equation}
Hence
$$\sum_{k,m=0}^{n-1}\mathbb E_{\bar\mu}[\tilde u.\kappa\circ\bar T^k\otimes\kappa\circ\bar T^m] =B_2^+(\tilde u)+O(\vartheta_0^{n/2}\Vert u\Vert_{(\xi)})\, ,$$
\begin{equation}
\sum_{k,m=0}^{n-1}\mathbb E_{\bar\mu}[\kappa\circ\bar T^k\otimes\hat\kappa\circ\bar T^m\tilde v\circ \bar T^n]\\
=B_2^-( v)+O(\vartheta_0^{n/2}\Vert v\Vert_{(\xi)}\, ,
\end{equation}
and
\begin{eqnarray*}
&\ &
   \sum_{k,m=0}^{n-1}\mathbb E_{\bar\mu}[\tilde u.\kappa\circ\bar T^k\otimes\kappa\circ\bar T^m.\tilde v\circ\bar T^n]\\
&=&\left(\sum_{k=0}^{n-1}
     \mathbb E_{\bar\mu}[\tilde u.\kappa^{\otimes 2}\circ \bar T^k.\tilde v\circ \bar T^n]+2\sum_{0\le k<m< n}\mathbb E_{\bar\mu}[\tilde u.\kappa\circ \bar T^k\otimes\kappa\circ \bar T^m.\tilde v\circ \bar T^n]\right)\\
&=&
 2\sum_{0\le k< m<n}\mathbb E_{\bar\mu}[\tilde u.(\kappa\circ \bar T^k)]\otimes\mathbb E_{\hat\mu}[\tilde v.\bar\kappa\circ\bar T^{n-m}]+O(\vartheta_0^{n/2}\Vert u\Vert_{(\xi)}\Vert v\Vert_{(\xi)})\\
&=&2B_1^+(u)\otimes B_1^-(v)+O(\vartheta_0^{n/2}\Vert u\Vert_{(\xi)}\Vert v\Vert_{(\xi)})  \, ,
\end{eqnarray*}
where we used the fact that $\mathbb E_{\bar\mu}[\tilde u.\kappa^{\otimes 2}\circ \bar T^k.\tilde v\circ \bar T^n]=O(\Vert u\Vert_{(\xi)}\Vert v\Vert_{(\xi)}\vartheta_0 ^{n/2})$.
\end{itemize}
Therefore we have proved \eqref{Pi"0}.
\item 
Let us prove \eqref{Pi'''0}. 
By bilinearity, we have
\begin{equation}\label{A311}
A_{3,n}(u,v)=A_{3,n}(\tilde u,\tilde v)+\mathbb E_{\bar \mu}[ u]A_{3,n}(\mathbf 1,\tilde v)+\mathbb E_{\bar \mu}[ v]A_{3,n}(\tilde u,\mathbf 1)+\mathbb 
E_{\bar\mu}[u]\mathbb E_{\bar\mu}[v]A_{3,n}(\mathbf 1,\mathbf 1).
\end{equation}
Note that 
$$A_{3,n}(\mathbf 1,\mathbf 1)=-i\mathbb E_{\bar\mu}[S_n^{\otimes 3}]=0.$$
since  $(S_n)_n$
has the same distribution as $(-S_n)_n$ (see the begining of the proof of Proposition \ref{pro:pertu2}).
We will use the following notations:
$c_{(k,m,r)}$ denotes the number of uples made of
$k,m,r$ (with their multiplicities) and we will write $\widetilde{\overbrace{F}}$ for $F-\mathbb E_{\bar\mu}[F]$
when $F$ is given by a long formula.
\begin{itemize}
\item We start with the study of $A_{3,n}(\tilde u,\mathbf 1)$.
\begin{eqnarray}
A_{3,n}(\tilde u,\mathbf 1)&=&-i\mathbb E_{\bar\mu}[\tilde u.S_n^{\otimes 3}]+3in\Sigma^2\otimes\mathbb E_{\bar\mu}[\tilde u.S_n]\nonumber\\
&=&-i\sum_{0\le k\le m\le r\le n-1}c_{k,m,r}\mathbb E_{\bar\mu}[\tilde u.\kappa\circ\bar T^k\otimes \kappa\circ\bar T^m\otimes\kappa\circ \bar T^r]+3in\Sigma^2\otimes\mathbb E_{\bar\mu}[\tilde u.S_n]\nonumber\\
&=&-i\sum_{0\le k\le m\le r\le n-1}c_{k,m,r}\mathbb E_{\bar\mu}[\tilde u.\kappa\circ\bar T^k]\otimes \mathbb E_{\bar\mu}[\kappa\circ\bar T^m\otimes\kappa\circ \bar T^r]+3in\Sigma^2\otimes\mathbb E_{\bar\mu}[\tilde u.S_n]\nonumber\\
&\ &-i\sum_{0\le k\le m\le r\le n-1}c_{k,m,r}\mathbb E_{\bar\mu}
\left[\widetilde{\overbrace{\tilde u.\kappa\circ\bar T^k}}\otimes\widetilde{\overbrace{\kappa\circ\bar T^m\otimes\kappa\circ\bar T^r}}\right]\nonumber
\end{eqnarray}
\begin{eqnarray}
&\ &A_{3,n}(\tilde u,\mathbf 1)\\
&=&-3i\sum_{k\ge 0}\sum_{m\in\mathbb Z}\max(0,n-|m|-k)
\mathbb E_{\bar\mu}[\tilde u.\kappa\circ\bar T^k]\otimes
\mathbb E_{\bar\mu}[\kappa.\kappa\circ\bar T^m]
+3in\Sigma^2\otimes\mathbb E_{\bar\mu}[\tilde u.S_n]\nonumber\\
&\ &-i\sum_{k,m,r=0}^{n-1}\mathbb E_{\bar\mu}
\left[\widetilde{\overbrace{\tilde u.\kappa\circ\bar T^{\min(k,m,r)}}}\otimes\widetilde{\overbrace{\kappa\circ\bar T^{med(k,m,r)}\otimes\kappa\circ\bar T^{\max(k,m,r)}}}\right]
\nonumber\\
&=&3i\sum_{k\ge 0}\sum_{m\in\mathbb Z}(|m|+k)
\mathbb E_{\bar\mu}[\tilde u.\kappa\circ\bar T^k]\otimes
\mathbb E_{\bar\mu}[\kappa.\kappa\circ\bar T^m]\nonumber\\
&\ &-3in\, (B_1^+(u)-\mathbb E_{\bar\mu}[\tilde u.S_n])\otimes \Sigma^2-iB_{3}^+(\tilde u)+O(\vartheta_0^{n/3}\Vert u\Vert_{(\xi)})\nonumber\\
\end{eqnarray}
and so
\begin{equation}
A_{3,n}(\tilde u,\mathbf 1)=-iB_3^+(\tilde u)+3iB_0^+(\tilde u)\otimes\Sigma^2+3iB_0\otimes B_1^+(\tilde u)\, .\label{A322}
\end{equation}
\item Analogously,
\begin{equation}
A_{3,n}(\mathbf 1,\tilde v)=-iB_3^-(\tilde v)+3iB_0^-(\tilde v)\otimes\Sigma^2+3iB_0\otimes B_1^-(\tilde v)\, .\label{A333}
\end{equation}
\item Finally
\begin{eqnarray*}
A_{3,n}(\tilde u,\tilde v)&=&-i \mathbb E_{\bar\mu}[\tilde u.S_n^{\otimes 3}.\tilde v \circ \bar T^n]+ 3i\, n \Sigma^2\otimes\mathbb E_{\bar\mu}[\tilde u.S_n.\tilde v \circ \bar T^n]\nonumber\\
&=&-i\sum_{k,m,r=0}^{n-1}\mathbb E_{\bar\mu}[
     \tilde u.\kappa\circ \bar T^k\otimes\kappa\circ\bar T^m\otimes\kappa\circ\bar T^r.\tilde v\circ\bar T^n]+3i n\Sigma^2\otimes\tilde A_{1,n}(\tilde u,\tilde v)\nonumber\\
&=&-i\sum_{k,m,r=0}^{n-1}\mathbb E_{\bar\mu}[
     \tilde u.\kappa\circ \bar T^k\otimes\kappa\circ\bar T^m\otimes\kappa\circ\bar T^r.\tilde v\circ\bar T^n]+O(n^2\vartheta_0^{n/2}\Vert u\Vert_{(\xi)}
\Vert v\Vert_{(\xi)})\, .
\end{eqnarray*}
Assume $0\le k\le m\le r\le n-1$.
Considering separately the cases $k\ge n/4$, $m-k\ge n/4$,
$r-m\ge n/4$ and $n-r\ge n/4$, we observe that
\begin{eqnarray}
&\ &\mathbb E_{\bar\mu}[
     \tilde u.\kappa\circ \hat T^k\otimes\kappa\circ\bar T^m.\otimes\kappa\circ\bar T^r.\tilde v\circ\bar T^n]\nonumber\\
&\ &=\mathbb E_{\bar\mu}[\tilde u.\kappa\circ \bar T^k]\otimes
\mathbb E_{\bar\mu}[\tilde v.\kappa\circ\bar T^{-(n-r)}\otimes  \kappa\circ\bar T^{-(n-m)}]\nonumber\\
&\ &+
\mathbb E_{\bar\mu}[\tilde v .\kappa\circ\bar T^{-(n-r)}]\otimes \mathbb E_{\bar\mu}[\tilde u.\kappa\circ\bar T^k\otimes\kappa\circ \bar T^m]+O(\vartheta_0^{n/4} \Vert v\Vert_{(\xi)}\, \Vert u\Vert_{(\xi)})\, .\label{Esp5termes}
\end{eqnarray}
And so 
\begin{equation}
A_{3,n}(\tilde u,\tilde v)
=-3iB_1^+(\tilde u)B_2^-(\tilde v)-3iB_1^-(\tilde v)B_2^+(\tilde u)\,
.
\end{equation}
\end{itemize}
This combined with \eqref{A311}, \eqref{A322} and \eqref{A333} leads to \eqref{Pi'''0}.
\item
It remains to prove \eqref{Pi40}. Observe first that
\begin{eqnarray}
 A_{4,n}(u,v)&=&(\lambda^{-n})^{(4)}_0\mathbb E_{\bar \mu}[\bar u]\mathbb E_{\bar \mu}[\bar v]+6n\Sigma^2\otimes\mathbb E_{\bar \mu}[u.S_n^{\otimes 2}.v\circ\bar T^n]+\mathbb E_{\bar \mu}[u.S_n^{\otimes 4}.v\circ\bar T^n]\nonumber\\
&=&(\lambda^{-n})^{(4)}_0\mathbb E_{\bar \mu}[\bar u]\mathbb E_{\bar \mu}[\bar v]+6n\Sigma^2\otimes \left(n\Sigma^2\mathbb E_{\bar\mu}[u]\mathbb E_{\bar\mu}[v]-A_{2,n}(u,v)\right)+\mathbb E_{\bar \mu}[u.S_n^{\otimes 4}.v\circ\bar T^n]\label{decompA4}
 \end{eqnarray}
where we used \eqref{formuleA2n}.
Note that
\begin{eqnarray}
\mathbb E_{\bar \mu}[u.S_n^{\otimes 4}.v\circ\bar T^n]
&=& \mathbb E_{\bar \mu}[\tilde u.S_n^{\otimes 4}.\tilde v\circ\bar T^n] +\mathbb E_{\bar \mu}[u]\mathbb E_{\bar \mu}[S_n^{\otimes 4}.\tilde v\circ\bar T^n]\nonumber\\
&\ &+\mathbb E_{\bar \mu}[v]\mathbb E_{\bar \mu}[\tilde u.S_n^{\otimes 4}]+\mathbb E_{\bar \mu}[u]\mathbb E_{\bar \mu}[v]\mathbb E_{\bar \mu}[S_n^{\otimes 4}].\label{A4bilin}
\end{eqnarray}
We now study separately each term of the right hand side
of this last formula. 
\begin{itemize}
\item First:
\begin{eqnarray}
&\ &\mathbb E_{\bar\mu}[\tilde u.S_n^{\otimes 4}   .\tilde v \circ \bar T^n]\nonumber\\
&\ &= \sum_{k,m,r,s=0}^{n-1}\mathbb E_{\hat\mu}[
     \tilde u.\kappa\circ \bar T^k\otimes\kappa\circ\bar T^m.\kappa\circ\bar T^r\otimes\kappa\circ\bar T^s.\tilde v\circ\bar T^n]\,\nonumber\\
&\ &= \sum_{0\le k\le m\le r\le s\le n-1}c_{(k,m,r,s)}
    \mathbb E_{\bar\mu}\left[
     \tilde u. \kappa\circ \bar T^k\otimes\widetilde{\overbrace{\kappa\otimes\kappa\circ\bar T^{r-m}}}\circ \bar T^m.\kappa\circ\bar T^s.\tilde v\circ\bar T^n\right]\,\nonumber\\
&\ & \ \ + \sum_{0\le k\le m\le r\le s\le n-1}c_{(k,m,r,s)}
    \mathbb E_{\bar\mu}[
     \tilde u.\kappa\circ \bar T^k.\kappa\circ\bar T^s.\tilde v\circ\bar T^n]\otimes\mathbb E_{\bar\mu}[\kappa\otimes\kappa\circ \bar T^{r-m}]\label{CCC2}
\end{eqnarray}
with $c_{(k,m,r,s)}$ the number of 4-uples made of $k,m,r,s$
(with the same multiplicities).
Due to \eqref{Esp4termes},
\begin{eqnarray}
&\ &\sum_{0\le k\le m\le r\le s\le n-1}c_{(k,m,r,s)}
    \mathbb E_{\bar\mu}[
     \tilde u.\kappa\circ \bar T^k\otimes\kappa\circ\bar T^s.\tilde v\circ\bar T^n]\otimes\mathbb E_{\bar\mu}[\kappa\otimes\kappa\circ\bar T^{r-m}]\nonumber\\
&\ & =\sum_{0\le k\le m\le r\le s\le n-1}c_{(k,m,r,s)}
    \mathbb E_{\bar\mu}[
     \tilde u.\kappa\circ \bar T^k]\otimes\mathbb E_{\bar \mu}[ \tilde v.\kappa\circ\bar T^{-(n-s)}]\otimes\mathbb E_{\bar\mu}[\kappa\otimes\kappa\circ\bar T^{r-m}]+O(n^4\vartheta_0^{n/3}\Vert u\Vert_{(\xi)}\Vert v\Vert_{(\xi)})\nonumber\\
&\ & =\sum_{k\ge 0}\mathbb E_{\bar\mu}[
     \tilde u.\kappa\circ \bar T^k]  \otimes\sum_{s\ge 1}
      \mathbb E_{\bar \mu}[\tilde v.\kappa \circ \bar T^{-s}]\otimes
\sum_{m=k}^{n-s}\sum_{r=m}^{n-s}c_{(k,m,r,n-s)}\mathbb E_{\bar\mu}[\kappa\otimes\kappa\circ \bar T^{r-m}]+O(n^4\vartheta_0^{n/3}\Vert u\Vert_{(\xi)}\Vert v\Vert_{(\xi)})\nonumber\\
&\ & =\sum_{k\ge 0}\mathbb E_{\bar\mu}[
     \tilde u.\kappa\circ \bar T^k]  \otimes\sum_{s\ge 1}
      \mathbb E_{\bar \mu}[\tilde v.\kappa\circ\bar T^{-s}]
  \otimes 12\mathbb E_{\bar\mu}[ S_{n-s-k+1}^{\otimes 2}]+O(n^4\vartheta_0^{n/3}\Vert u\Vert_{(\xi)}\Vert v\Vert_{(\xi)})\nonumber\\
&\ & =\sum_{k\ge 0}\mathbb E_{\bar\mu}[
     \tilde u.\kappa\circ \hat T^k]  \otimes\sum_{s\ge 1}
      \mathbb E_{\bar \mu}[\tilde v.\kappa\circ\bar T^{-s}]12((n-s-k+1)\Sigma^2-\sum_{r\in\mathbb Z}|r|\mathbb E_{\hat\mu}[\hat\kappa\otimes \hat\kappa\circ \hat T^r]+O(n^4\vartheta_0^{n/3}\Vert u\Vert_{(\xi)}\Vert v\Vert_{(\xi)})\nonumber\\
&\ &=12 B_1^+(\tilde u)B_1^-(\tilde v)\left(n\Sigma^2-\sum_{r\in\mathbb Z}|r|\mathbb E_{\hat\mu}[\hat\kappa\otimes \hat\kappa\circ \hat T^r]\right)\nonumber\\
&\ &-12
\sum_{k\ge 0}\mathbb E_{\bar\mu}[
     \tilde u.\kappa\circ \bar T^k]\otimes  \sum_{s\ge 1}
      \mathbb E_{\bar \mu}[\tilde v.\kappa\circ\bar T^{-s}](s+k-1)\otimes\Sigma^2 
+O(n^4\vartheta_0^{n/3}\Vert u\Vert_{(\xi)}\Vert v\Vert_{(\xi)})\, .\label{CCC3}
\end{eqnarray}
But, on the other hand, treating separately the cases $k\ge n/5$, $m-k\ge n/5$,
$r-m\ge n/5$, $s-r\ge n/5$ and $n-s\ge n/5$, we obtain that, for every $0\le k\le m\le r\le s\le n$,
\begin{eqnarray}
&\ &    \mathbb E_{\bar\mu}\left[
     \tilde u.\kappa\circ \bar T^k\otimes\widetilde{\overbrace{\kappa\otimes\kappa\circ\bar T^{r-m}}}\circ \bar T^m\otimes\kappa\circ\bar T^s.\tilde v\circ\bar T^n\right]\nonumber\\
&\ &=    \mathbb E_{\bar\mu}[
     \tilde u.\kappa\circ \bar T^k]\otimes\mathbb E_{\bar\mu}[\widetilde{\overbrace{\kappa\otimes\kappa\circ\bar T^{r-m}}}\circ \bar T^m\otimes\kappa\circ\bar T^s.\tilde v\circ\bar T^n]\nonumber\\
&\ & +\    \mathbb E_{\bar\mu}[
     \tilde u.\kappa\circ \bar T^k\otimes\hat\kappa\circ \hat T^m]\otimes \mathbb E_{\bar\mu}[\kappa\circ\bar T^{r}\otimes\kappa\circ\bar T^s.\tilde v\circ\bar T^n]\nonumber\\
&\ & +\   \mathbb E_{\bar\mu}[
     \tilde u.\kappa\circ \bar T^k\otimes\widetilde{\overbrace{\kappa\otimes\kappa\circ\bar T^{r-m}}}\circ \bar T^m] \otimes\mathbb E_{\bar\mu}[\kappa\circ\bar T^s.\tilde v\circ\bar T^n]+O(\vartheta_0^{n/5}\Vert u\Vert_{(\xi)}\Vert v\Vert_{(\xi)}).
\end{eqnarray}
Due to \eqref{Esp3termes},
$$ \mathbb E_{\bar\mu}\left[\widetilde{\overbrace{\kappa.\kappa\circ\bar T^{r-m}}}\circ \bar T^m.\kappa\circ\bar T^s.\tilde v\circ\bar T^n\right]=O(\vartheta_0^{n-m}\Vert u\Vert_{(\xi)}\Vert v\Vert_{(\xi)})\, ,$$
$$
\mathbb E_{\bar\mu}[
     \tilde u.\kappa\circ \bar T^k.\widetilde{\overbrace{\kappa.\kappa\circ\bar T^{r-m}}}\circ \bar T^m]\mathbb E_{\bar\mu}[\kappa\circ\bar T^s.\tilde v\circ\bar T^n]=O(\vartheta_0^{m}\vartheta_0^{n-s}\Vert u\Vert_{(\xi)}\Vert v\Vert_{(\xi)})\, ,$$
$$\mathbb E_{\bar\mu}[
     \tilde u.\kappa\circ \bar T^k\otimes\kappa\circ \bar T^m]=O(\vartheta_0^{m}\Vert u\Vert_{(\xi)}).$$
Therefore
\begin{eqnarray}
&\ &\sum_{0\le k\le m\le r\le s\le n-1}c_{(k,m,r,s)}
    \mathbb E_{\bar\mu}[
     \tilde u.\kappa\circ \bar T^k\otimes\widetilde{\overbrace{\kappa\circ \bar T^m\otimes\kappa\circ\bar T^{m}}}\otimes\kappa\circ\bar T^s.\tilde v\circ\bar T^n]\nonumber\\
&\ &=4\sum_{k\ge 0}    \mathbb E_{\bar\mu}[
     \tilde u.\kappa\circ \bar T^k]B_3^-(\tilde v)+4 B_3^+(\tilde u)\otimes\sum_{s\ge 1}\mathbb E_{\bar\mu}[\kappa\circ \bar T^{-s}.\tilde v]\nonumber\\
&\ & +6 \sum_{m, k\ge 0}    \mathbb E_{\bar\mu}[
     \tilde u.\kappa\circ \bar T^k\otimes\kappa\circ \bar T^m]\otimes \sum_{r,s\ge 1}
\mathbb E[\tilde v. \kappa\circ\bar T^{-r}\otimes \kappa\circ\bar T^{-s}]\nonumber\\
&\ &  +O\left(\vartheta_0^{n/5}\Vert u\Vert_{(\xi)}\Vert v\Vert_{(\xi)}\right)\, .\label{CCC4}
\end{eqnarray}
Putting together \eqref{decompA4}, \eqref{CCC2}, \eqref{CCC3} and \eqref{CCC4} leads to 
\begin{eqnarray}
A_{4,n}(\tilde u,\tilde v)&=&-12\sum_{k\ge 0}\mathbb E_{\hat\mu}[
     \tilde u.\kappa\circ \hat T^k]  \sum_{s\ge 1}
      \mathbb E_{\hat \mu}[ \kappa\circ\bar T^{-s}.\tilde v] (s+k-1)\otimes\Sigma^2\nonumber\\
&\ &+4B_1^+(u)\otimes B_3^-(\tilde v)+4\  B_1^-(v)\otimes B_3^+(\tilde u)\nonumber\\
&\ & +6 B_2^+(u)\otimes B_2^-(v)-12B_1^+(\tilde u)\otimes B_1^-(\tilde v)\otimes B_0 +O\left(\vartheta_0^{n/5}\Vert u\Vert_{(\xi)}\Vert v\Vert_{(\xi)}\right)\,\label{A4tilde}
\end{eqnarray}
\item Second:
\begin{equation}
\mathbb E_{\bar \mu}[\tilde u.S_n^{\otimes 4}]
=\sum_{0\le k\le m\le r\le s\le n-1}c_{(k,m,r,s)}\mathbb E_{\bar\mu}\left[\tilde u.\kappa\circ\bar T^k\otimes\kappa\circ\bar T^m\otimes\kappa\circ\bar T^r\otimes\kappa\circ\bar T^s\right].
\end{equation}
But, due to \eqref{Esp5termes}, for $0\le k\le m\le r\le s\le n-1$,
we have
\begin{eqnarray*}
&\ &\mathbb E_{\bar\mu} [ \tilde u.\kappa\circ \bar T^k\otimes\kappa\circ \bar T^m\otimes
\kappa
\circ\bar T^r\otimes\kappa\circ \bar T^s]\\
&\ &=\mathbb E_{\bar\mu}[\tilde u.\kappa \circ \bar T^k]\otimes\mathbb E_{\bar\mu}[ \kappa\otimes\kappa\circ\bar T^{r-m}\otimes \kappa\circ \bar T^{s-m}] \\
&\ &+\mathbb E_{\bar\mu}[\tilde u.\kappa \circ \bar T^k\otimes \kappa\circ\bar T^m]\otimes\mathbb E_{\bar\mu}[ \kappa\otimes \kappa\circ \bar T^{s-r}]+O(\Vert u\Vert_{(\xi)}\Vert v\Vert_{(\xi)}\vartheta_0^{s/3})\, .
\end{eqnarray*}
Therefore
\begin{eqnarray*}
&\ &\mathbb E_{\bar \mu}[\tilde u.S_n^{\otimes 4}]=4\sum_{k\ge 0}\mathbb E_{\bar\mu}[S_{n-k}^{\otimes 3}]\\
&\ &+6\sum_{k,m\ge 0}\sum_{r\in\mathbb Z}\max(0,(n-\max(k,m)-|r|))
\mathbb E_{\bar\mu}[\tilde u.\kappa \circ \bar T^k\otimes \kappa\circ\bar T^m]\otimes\mathbb E_{\bar\mu}[ \kappa\otimes \kappa\circ \bar T^{r}]\\
&=&6nB_2^+(\tilde u)\otimes\Sigma^2-6
\sum_{k,m\ge 0}\sum_{r\in\mathbb Z}(\max(k,m)+|r|)
\mathbb E_{\bar\mu}[\tilde u.\kappa \circ \bar T^k\otimes \kappa\circ\bar T^m]\otimes\mathbb E_{\bar\mu}[ \kappa\otimes \kappa\circ \bar T^{r}]
\end{eqnarray*}
since $\mathbb E_{\bar\mu}[S_n^{\otimes 3}]=0$.
It comes
\begin{equation}
\mathbb E_{\bar \mu}[\tilde u.S_n^{\otimes 4}]
=6nB_2^+(\tilde u)\otimes\Sigma^2-6 
B_{0,2}^+(\tilde u)\otimes \Sigma^2-6 B_2^+(\tilde u)\otimes B_0+O(\vartheta_0^{n/2})\label{A4u1}
\end{equation}
\item Analogously, 
\begin{equation}
\mathbb E_{\bar \mu}[\tilde v\circ\bar T^n.S_n^{\otimes 4}]
=6nB_2^-(\tilde v)\otimes\Sigma^2-6 
B_{0,2}^-(\tilde v)\otimes \Sigma^2-6 B_2^-(\tilde v)\otimes B_0+O(\vartheta_0^{n/2})\, .\label{A41v}
\end{equation}
\end{itemize}
Formula \eqref{formuleA4} follows from \eqref{A4bilin}, \eqref{A4tilde}, \eqref{A4u1} and \eqref{A41v}.
\end{itemize}
\end{proof}
\begin{proposition}\label{lambda04}
The fourth derivatives of $\lambda$ at $0$ are given by
\[
\lambda_0^{(4)}=\lim_{n\rightarrow +\infty}\frac{\mathbb E_{\bar\mu}[S_n^{\otimes 4}]-3n^2(\Sigma^2)^{\otimes 2}}{n}+3(\Sigma^2)^{\otimes 2}+6\Sigma^2\otimes B_0\, .
\]
\end{proposition}
\begin{proof}
Derivating four times $\mathbb E_{\bar\mu}[e^{it*S_n}]=\lambda_t^n \mathbb E_{\bar\mu}[e^{it*S_n}/\lambda_t^n]$ leads to
\begin{eqnarray*}
\mathbb E_{\bar\mu}[S_n^{\otimes 4}]&=&(\lambda^n)_0^{(4)}
+6(\lambda^n)_0^{(2)}\otimes A_{2,n}(\mathbf 1,\mathbf 1)+A_{4,n}(\mathbf 1,\mathbf 1)\\
&=& n\lambda_0^{(4)}+3n(n-1)(\lambda_0^{(2)})^{\otimes 2}+6n\lambda_0^{(2)}\otimes A_{2,n}(\mathbf 1,\mathbf 1)
   +A_{4,n}(\mathbf 1,\mathbf 1)\, ,
\end{eqnarray*}
and we conclude due to \eqref{Amlim} and due to $\lambda_0^{(2)}=-\Sigma^2$ (coming from Item (iii) of Proposition \ref{pro:pertu}).
\end{proof}

\noindent{\bf Acknowledgment.\/}{
The author wishes to thank Damien Thomine for interesting discussions 
having led to an improvement of the assumption for the mixing result in the infinite horizon billiard case.
}
\end{appendix}


\end{document}